\documentclass[11pt]{article}
\usepackage{amsmath}
\usepackage{amsfonts}
\usepackage{amssymb}
\usepackage{mathrsfs}
\usepackage{amsthm}
\usepackage{palatino}
\usepackage[margin=2.5cm, vmargin={1.5cm}]{geometry}

\usepackage{dcolumn}

\usepackage{times}
\usepackage{graphicx}
\usepackage{xcolor}

\usepackage{pstricks}
\usepackage{pst-plot}
\usepackage{pst-grad}

\renewcommand\footnotemark{}

\begin{document}

\title{Stirling's approximation and a hidden link between two of Ramanujan's approximations}

\author{
  Cormac ~O'Sullivan\footnote{{\it Date:}Feb 9, 2023.
\newline \indent \ \ \
  \newline \indent \ \ \
Support for this project was provided by a PSC-CUNY Award, jointly funded by The Professional Staff Congress and The City
\newline \indent \ \ \
University of New York.}
  }

\date{}

\maketitle

\def\s#1#2{\langle \,#1 , #2 \,\rangle}

\def\F{{\frak F}}
\def\C{{\mathbb C}}
\def\R{{\mathbb R}}
\def\Z{{\mathbb Z}}
\def\Q{{\mathbb Q}}
\def\N{{\mathbb N}}
\def\G{{\Gamma}}
\def\GH{{\G \backslash \H}}
\def\g{{\gamma}}
\def\L{{\Lambda}}
\def\ee{{\varepsilon}}
\def\K{{\mathcal K}}
\def\Re{\mathrm{Re}}
\def\Im{\mathrm{Im}}
\def\PSL{\mathrm{PSL}}
\def\SL{\mathrm{SL}}
\def\Vol{\operatorname{Vol}}
\def\lqs{\leqslant}
\def\gqs{\geqslant}
\def\sgn{\operatorname{sgn}}
\def\res{\operatornamewithlimits{Res}}
\def\li{\operatorname{Li_2}}
\def\lip{\operatorname{Li}'_2}
\def\pl{\operatorname{Li}}

\def\ei{\mathrm{Ei}}

\def\clp{\operatorname{Cl}'_2}
\def\clpp{\operatorname{Cl}''_2}
\def\farey{\mathscr F}

\def\dm{{\mathcal A}}
\def\ov{{\overline{p}}}
\def\ja{{K}}

\def\nb{{\mathcal B}}
\def\cc{{\mathcal C}}
\def\nd{{\mathcal D}}

\newcommand{\stira}[2]{{\genfrac{[}{]}{0pt}{}{#1}{#2}}}
\newcommand{\stirb}[2]{{\genfrac{\{}{\}}{0pt}{}{#1}{#2}}}
\newcommand{\eu}[2]{{\left\langle\!\! \genfrac{\langle}{\rangle}{0pt}{}{#1}{#2}\!\!\right\rangle}}
\newcommand{\eud}[2]{{\big\langle\! \genfrac{\langle}{\rangle}{0pt}{}{#1}{#2}\!\big\rangle}}
\newcommand{\norm}[1]{\left\lVert #1 \right\rVert}

\newcommand{\e}{\eqref}
\newcommand{\bo}[1]{O\left( #1 \right)}


\newtheorem{theorem}{Theorem}[section]
\newtheorem{lemma}[theorem]{Lemma}
\newtheorem{prop}[theorem]{Proposition}
\newtheorem{conj}[theorem]{Conjecture}
\newtheorem{cor}[theorem]{Corollary}
\newtheorem{assume}[theorem]{Assumptions}
\newtheorem{adef}[theorem]{Definition}

\numberwithin{figure}{section}
\numberwithin{table}{section}


\newcounter{counrem}
\newtheorem{remark}[counrem]{Remark}

\renewcommand{\labelenumi}{(\roman{enumi})}
\newcommand{\spr}[2]{\sideset{}{_{#2}^{-1}}{\textstyle \prod}({#1})}
\newcommand{\spn}[2]{\sideset{}{_{#2}}{\textstyle \prod}({#1})}

\numberwithin{equation}{section}

\let\originalleft\left
\let\originalright\right
\renewcommand{\left}{\mathopen{}\mathclose\bgroup\originalleft}
\renewcommand{\right}{\aftergroup\egroup\originalright}

\bibliographystyle{alpha}

\begin{abstract}
A conjectured relation between Ramanujan's asymptotic approximations to the exponential function and the exponential integral is established. The proof involves Stirling numbers, second-order Eulerian numbers, modifications of both of these, and Stirling's approximation to the gamma function. Our work provides new information about the coefficients in Stirling's approximation and their connection to  Ramanujan's approximation coefficients. A  more analytic second proof of the main result is also included in an appendix.
\end{abstract}

\section{Introduction}
From about 1903 to 1914 Ramanujan listed his many mathematical discoveries in two main notebooks.
In Entry 48 of Chapter 12 of his second notebook,  he considered the Taylor expansion of $e^n$. Writing
\begin{equation}\label{rb}
  1+\frac{n}{1!}+\frac{n^2}{2!}+ \cdots +\frac{n^{n-1}}{(n-1)!} +  \frac{n^n}{n!}\theta_n = \frac{e^n}2
\end{equation}
to define $\theta_n$, he computed an asymptotic expansion that is equivalent to
\begin{equation}\label{rb2}
   \theta_n  = \frac{1}3+ \frac{4}{135 n}-\frac{8}{2835 n^2}-\frac{16}{8505 n^3}+O\left( \frac{1}{n^4}\right),
\end{equation}
as $n \to \infty$.
His next result, Entry 49, is quite surprising as it uses the reciprocals of the terms on the left of \e{rb}. It can be expressed as
\begin{equation}\label{ei}
  1+\frac{1!}{n}+\frac{2!}{n^2}+ \cdots +\frac{(n-1)!}{n^{n-1}} +  \frac{n!}{n^n}\Psi_n = \frac{n \ei(n)}{e^{n}},
\end{equation}
for $\ei(n)$ the exponential integral, and the asymptotic expansion of $\Psi_n$ is given as
\begin{equation}\label{ei2}
   \Psi_n  = -\frac{1}3+ \frac{4}{135 n}+\frac{8}{2835 n^2}-\frac{16}{8505 n^3}+O\left( \frac{1}{n^4}\right).
\end{equation}
For Berndt's discussion of these entries see \cite[pp. 181--184]{Ber89}. The expansions \e{rb2} and \e{ei2} may be continued to any order and we label the coefficients as $\rho_r$ and $\psi_r$ respectively:
\begin{equation} \label{iou}
  \theta_n \sim \sum_{r\gqs 0} \frac{\rho_r}{n^r}, \qquad \Psi_n \sim \sum_{r\gqs 0} \frac{\psi_r}{n^r}.
\end{equation}
 Ramanujan in fact found five terms in the expansion \e{rb2} and three terms in \e{ei2}. He must have noticed the agreement of the coefficients, though only had the evidence of the first three to compare. 


\begin{theorem} \label{ky}
 We have $\psi_r = (-1)^{r+1}\rho_r$ for all $r\gqs 0$.
\end{theorem}

This relation is the hidden link referred to in the title and does not appear to have been discussed before it recently became Conjecture 8.4 of \cite{OSra}. The proof we have found   also seemed quite well concealed, and requires a third sequence $\g_r$ from possibly the most important asymptotic expansion of all:
\begin{equation} \label{gmx}
  \G(n+1)= \sqrt{2\pi n}\left(\frac{n}{e}\right)^n \left(\g_0+\frac{\g_1}{n}+\frac{\g_2}{n^2}+ \cdots + \frac{\g_{R-1}}{n^{R-1}} +O\left(\frac{1}{n^{R}}\right)\right),
\end{equation}
as $n \to \infty$.
This is Stirling's approximation for the gamma function\footnote{Ramanujan also had his own version of Stirling's approximation; see \cite{Ka01} for example}, and in the course of our work we will gain more information about the Stirling coefficients $\g_r$. This sequence begins
\begin{equation*}
  \g_0 = 1, \quad \g_1 = \frac{1}{12}, \quad \g_2 = \frac{1}{288}, \quad \g_3 = -\frac{139}{51840}.
\end{equation*}
 In fact, up to a normalization, the numbers $\g_r$ and $\rho_r$ alternate in the series expansion of Watson's function $U(t)$ from \cite{Wa29}, which for $t\gqs 0$ satisfies
\begin{equation}\label{wu}
  U(0)=1, \quad U(t)\gqs 1, \quad U(t)+U'(t)=U(t)U'(t).
\end{equation}
See Section \ref{proo} where many properties of the sequences $\g_r$ and $\rho_r$ are assembled.

Other key ingredients we will need are the Bernoulli numbers $B_n$, the Stirling cycle numbers and the second-order Eulerian numbers $\eud{n}{k}$. These last may be defined recursively, for integers $n$ and $k$ with $n \gqs 0$, by the initial condition $\eud{0}{k}=\delta_{k,0}$, with the Kronecker delta, and  the relation 
\begin{equation} \label{eur}
  \eu{n+1}{k} = (k+1)\eu{n}{k}+(2n+1-k)\eu{n}{k-1},
\end{equation}
as in \cite[Sect. 6.2]{Knu}.
In 2010,  Majer raised the question  in \cite{mo} if
\begin{equation}\label{move}
   \sum_{k= 0}^n (-1)^k \eu{n}{k}\binom{2n+1}{k+1}^{-1} = 2B_{n+1} \qquad \text{for all $n\gqs 0$}
\end{equation}
is true. This question was also highlighted in \cite{RU19}. The first proof of \e{move} has only recently been provided by Fu in \cite{Fu}, where a connection between N\"orlund polynomials and second-order Eulerian numbers is found. One of the steps we use to establish Theorem \ref{ky} requires \e{move}, and we find a different natural proof of it in Theorem \ref{mov}.

Our proof of Theorem \ref{ky} is essentially combinatorial. During the reviewing process for this paper, one of the anonymous referees provided an elegant analytic proof. This employs  the Laplace transform and makes use of the properties of $U(t)$ and various special functions. We give this alternate approach  in an appendix and warmly thank the referee for providing their insights  and a contrasting viewpoint.

\section{Initial results} \label{ini}
We begin by quoting the results needed from \cite{OSra}.
The asymptotic expansion of $\theta_n$ is established there  using Perron's saddle-point method. This is the $v=0$ case of \cite[Thm. 4.2]{OSra} with the coefficients $\rho_r$ given explicitly.  It is convenient to put $\hat\rho_r:=\rho_r -\delta_{r,0}$ so that  $\hat\rho_0=-2/3$. The asymptotic expansion
of $\G(n+1)$ is shown similarly in \cite[Cor. 4.3]{OSra}, and rewriting the formulas  slightly produces
\begin{align}
  \hat\rho_r & = (2r)!!\sum_{k=0}^{2r+1} \binom{-r-1}{k} 2^k \dm_{2r+1,k}\left( \frac{1}{3}, \frac{1}{4}, \frac{1}{5}, \dots \right), \label{vbw}\\
  \g_r & = (2r-1)!!\sum_{k=0}^{2r}  \binom{-r-1/2}{k} 2^k \dm_{2r,k}\left( \frac{1}{3}, \frac{1}{4}, \frac{1}{5}, \dots \right), \label{vbw2}
\end{align}
where the usual double factorial notation  has
\begin{equation} \label{doub}
  n!! := \begin{cases}
  n(n-2) \cdots 5 \cdot 3 \cdot 1 & \text{ if $n$ is odd};\\
  n(n-2) \cdots 6 \cdot 4 \cdot 2 & \text{ if $n$ is even},
  \end{cases}
\end{equation}
for $n\gqs 1$, with $0!!=(-1)!!=1$.
 The De Moivre polynomial $\dm_{n,k}(a_1, a_2, a_3, \dots)$ means, in general, the coefficient of $x^n$ in the series expansion of
 \begin{equation} \label{bell}
    \left( a_1 x +a_2 x^2+ a_3 x^3+ \cdots \right)^k,
\end{equation}
and here $a_j$ is $1/(j+2)$ for $j \gqs 1$.
Additionally, the $v=0$ case of \cite[Prop. 8.3]{OSra} finds
\begin{equation} \label{gps}
  \G(n+1)  \Psi_n = \sqrt{2\pi n}\left(\frac{n}{e}\right)^n \left(\tau_0+\frac{\tau_1}{n}+\frac{\tau_2}{n^2}+ \cdots + \frac{\tau_{R-1}}{n^{R-1}} +O\left(\frac{1}{n^{R}}\right)\right)
\end{equation}
for
\begin{equation}\label{taur}
   \tau_r  = (2r-1)!! \sum_{k=0}^{2r+1}  \binom{-r-1/2}{k} 2^k \dm_{2r+1,k}\left( \frac{1}{3}, \frac{1}{4}, \frac{1}{5}, \dots \right).
\end{equation}
Hence \e{gmx} and \e{gps} imply
\begin{equation}\label{ins}
  \sum_{j=0}^r  \psi_j \g_{r-j} = \tau_r.
\end{equation}

Our next goal is to express Theorem \ref{ky} in terms of the quantities $\hat\rho_r$, $\g_r$ and $\tau_r$ that we have formulas for. This is the first of a sequence of equivalences that we will use.

\begin{prop} \label{roe}
Theorem \ref{ky} is equivalent to
\begin{equation}\label{alx}
  \g_r +\tau_r +\sum_{j=0}^r (-1)^j \hat\rho_j \g_{r-j} =0 \qquad \text{for all $r\gqs 0$.}
\end{equation}
\end{prop}

Before giving the proof, some properties of the Stirling coefficients are required. Define the formal series
\begin{equation} \label{gz}
  g(z):=\g_0+\g_1 z+\g_2 z^2 +\g_3 z^3+ \cdots .
\end{equation}
Euler-Maclaurin summation  may be used to give the well-known asymptotic expansion of $\log \G(z)$, as in \cite[Appendix D]{AAR} for example. This produces
\begin{equation} \label{logg}
  \log g(z) =\frac{B_2}{2\cdot 1} z + \frac{B_4}{4\cdot 3} z^3 + \frac{B_6}{6\cdot 5} z^5 + \cdots .
\end{equation}
(The Bernoulli numbers $B_j$ for odd $j\gqs 3$ vanish.) Therefore
\begin{align}\label{hx}
  \frac{g'(z)}{g(z)} & = \frac{B_2}{2}  + \frac{B_4}{4} z^2 + \frac{B_6}{6} z^4 + \cdots
  ,\\
  r \g_r & = \sum_{j=0}^{r-1} \frac{B_{j+2}}{j+2} \g_{r-1-j} \qquad (r\gqs 0). \label{hy}
\end{align}
Since \e{logg} has only odd powers,
\begin{equation*}
  1/g(z) = \exp(-\log g(z))=\exp(\log g(-z))= g(-z).
\end{equation*}
Then $g(-z) g(z)=1$ is equivalent to
\begin{equation} \label{ggd}
  \sum_{j=0}^r (-1)^j \g_j \g_{r-j} = \delta_{r,0},
\end{equation}
and indeed, for any sequence $a_j$,
\begin{equation} \label{iff}
  \sum_{j=0}^r  \g_j a_{r-j} = b_r \quad \iff \quad \sum_{j=0}^r (-1)^j \g_j b_{r-j} = a_r.
\end{equation}
Together, \e{hy} and \e{iff} imply an identity we will need in Section \ref{ber}:
\begin{equation}\label{iff2}
  \sum_{j=0}^{r} (-1)^{j} (j+1) \g_{j+1}\g_{r-j} =  \frac{B_{r+2}}{r+2} \qquad(r\gqs 0).
\end{equation}

\begin{proof}[Proof of Proposition \ref{roe}]
Using \e{ggd} and \e{ins} we have
\begin{align*}
  (-1)^{r+1}\rho_r   = \psi_r  & \iff \sum_{j=0}^r  (-1)^{j+1} \rho_j \g_{r-j} = \sum_{j=0}^r  \psi_j \g_{r-j}\\
 & \iff \sum_{j=0}^r  (-1)^{j+1} \rho_j \g_{r-j} = \tau_r,
\end{align*}
and this last is equivalent to \e{alx}.
\end{proof}

\section{Lagrange inversion} \label{ut}
A formal power series $F(x)=\sum_{j=0}^\infty a_j x^j \in \C[[x]]$ with $a_0=0$ has a compositional inverse $G(x)\in \C[[x]]$ if and only if $a_1\neq 0$. This means $F(G(x))=x=G(F(x))$, where $G(x)$ is  necessarily unique with no constant term. See \cite[Sect. 3]{odm}, for example, for these basic results. The following general form of Lagrange inversion is Theorem 2.1.1 and (2.1.7) of \cite{Ge16}, and proved there in a variety of ways that include Jacobi's classic short proof. Note that a formal Laurent series takes the form $\sum_{j=n_0}^\infty a_j x^j$, for some $n_0\in \Z$, and the operator $[x^n]$ extracts the coefficient of $x^n$ from a Laurent series.

\begin{theorem}[Lagrange inversion]
Let $F$ be a formal power series with no constant term and compositional inverse $G$. Then for all integers $n$ and all formal Laurent series $\phi$,  we have
\begin{equation}\label{lin}
  [x^n] \phi(G(x)) = [x^{-1}] \frac{\phi(x) F'(x)}{F(x)^{n+1}}.
\end{equation}
\end{theorem}

Taking $\phi(x)=x^r$ and $\phi(x)=F(x)/(x F'(x))$ in \e{lin} gives, respectively,
\begin{align}\label{xla}
  n[x^n] G(x)^r & = r[x^{n-r}]\left(\frac x{F(x)}\right)^n \qquad (n, r \in \Z, \ r\neq 0),\\
\label{xla2}
  [x^{n-1}] \frac{G'(x)}{G(x)} & = [x^{n}]\left(\frac x{F(x)}\right)^n \qquad (n \in \Z).
\end{align}
We will use \e{xla} and \e{xla2} below after expressing $\hat\rho_r$, $\g_r$ and $\tau_r$ in terms of power series.
With \e{vbw} we have
\begin{align*}
  \hat\rho_r & = (2r)!!\sum_{k=0}^{2r+1} \binom{-r-1}{k} 2^k \left[ x^{2r+1}\right] S(x)^k \\
  & = (2r)!! \left[ x^{2r+1}\right] \sum_{k=0}^{2r+1} \binom{-r-1}{k} (2S(x))^k
   = (2r)!! \left[ x^{2r+1}\right] (1+2S(x))^{-r-1},
\end{align*}
for the series
\begin{equation*}
  S(x):= \frac{x}{3}+ \frac{x^2}{4}+\frac{x^3}{5}+ \cdots = \frac{-\log(1-x)-x}{x^2} - \frac 12.
\end{equation*}
Arguing similarly for $\g_r$, $\tau_r$ and letting
\begin{equation}
  V(x):= (1+2S(x))^{1/2}= \left( \frac{-2}{x^2}\left( \log(1-x)+x\right)\right)^{1/2} =
  1+\frac{x}{3}+\frac{7 x^2}{36}+\frac{73 x^3}{540}+ \cdots, \label{vx}
\end{equation}
shows
\begin{align}
   \hat\rho_r  & = (2r)!! \left[ x^{2r+1}\right] V(x)^{-2r-2}, \label{vv1}\\
   \g_r & = (2r-1)!!  \left[ x^{2r}\right] V(x)^{-2r-1}, \label{vv2}\\
   \tau_r & = (2r-1)!!  \left[ x^{2r+1}\right] V(x)^{-2r-1}. \label{vv3}
\end{align}
To go further, we express \e{vv1} -- \e{vv3} using the compositional inverse of $x V(x)$, which  may be written as $x V^*(x)$.
For all $n \in \Z$, by \e{xla} with $r=1$ and \e{xla2},
\begin{align*}
  \left[ x^{n}\right] V(x)^{-n-1} & = (n+1) \left[ x^{n}\right] V^*(x), \\
  \left[ x^{n}\right] V(x)^{-n} & =  \delta_{n,0}+\left[ x^{n-1}\right] V^*(x)'/V^*(x).
\end{align*}
 Hence
\begin{equation} \label{vs}
  V^*(x) = 1-\frac{x}{3}+\frac{x^2}{36}+\frac{x^3}{270}+\frac{x^4}{4320}- \cdots ,
\end{equation}
and
\begin{align}
   \hat\rho_r  & = (2r+2)!! \left[ x^{2r+1}\right] V^*(x), \label{vb1}\\
   \g_r & = (2r+1)!!  \left[ x^{2r}\right] V^*(x), \label{vb2}\\
   \tau_r & = (2r-1)!!  \left[ x^{2r}\right]V^*(x)'/V^*(x). \label{vb3}
\end{align}

\section{Watson's $U$ and $u$ functions} \label{Uu}
If $x V(x)=y$ then it follows from \e{vx} that $\exp(-y^2/2)=(1-x)\exp(x)$. To get our final simplified form, let
\begin{equation*}
  t=y^2/2, \quad u=1-x, \quad \text{so that} \quad e^{-t}=u e^{1-u}.
\end{equation*}
This means that $u=1\pm \sqrt{2t} V^*(\mp \sqrt{2t})$ and we may write these two cases as
\begin{align*}
  U(t) & := 1 + \sqrt{2t} V^*\left(-\sqrt{2t}\right) = 1+\sqrt{2} t^{1/2} +\frac 23 t + \frac{\sqrt 2}{18} t^{3/2} -\frac 2{135} t^2+ \cdots, \\
  u(t) & := 1 - \sqrt{2t} V^*\left(+\sqrt{2t}\right) = 1-\sqrt{2} t^{1/2} +\frac 23 t - \frac{\sqrt 2}{18} t^{3/2} -\frac 2{135} t^2+ \cdots.
\end{align*}
Label the coefficients as $c_j$ so that
\begin{equation*}
  U(t)=\sum_{j=0}^\infty c_j t^{j/2}, \qquad u(t)=\sum_{j=0}^\infty (-1)^j c_j t^{j/2}.
\end{equation*}
We will only need $U(t)$ and often just write $U$. Then
\begin{equation*}
  e^{-t}=U e^{1-U} \implies U+U'=UU',
\end{equation*}
and consequently the coefficients of $U$ satisfy the recursive definition
\begin{equation} \label{crec}
 c_m=0 \text{ when } m < 0, \quad c_0=1,  \quad c_1 >0, \quad \sum_{x+y=n} x c_x c_y =   n c_n + 2 c_{n-2},
\end{equation}
with summation over nonnegative integers $x$ and $y$.
In this paper, (outside of the Appendix), $U$ may be considered as a formal series. The properties of $U(t)$ and $u(t)$ as functions of $t$  are developed by Watson in \cite{Wa29}, where he proves bounds for $\theta_n$ in \e{rb} that were claimed by Ramanujan. Also the connection to branches of the Lambert $W$ function is seen in \cite[Eq. (3.2)]{Vo08}:
for $t\gqs 0$,
\begin{equation*}
  U(t)=-W_{-1}(-e^{-1-t}), \qquad u(t)=-W_0(-e^{-1-t}).
\end{equation*}

Next we express everything in terms of the coefficients $c_j$.

\begin{prop} \label{lus}
We have
\begin{gather} \label{rg}
  \hat\rho_r =-  \frac{(2r+2)!! c_{2r+2}}{2^{(2r+2)/2}}, \qquad \g_r =  \frac{(2r+1)!! c_{2r+1}}{2^{(2r+1)/2}},\\
  \tau_r = \g_r- \frac{(2r-1)!!}{2^{(2r-1)/2}} \frac 14 \sum_{x+y=2r+3} x c_x \cdot y c_y. \label{rgg}
\end{gather}
\end{prop}
\begin{proof}
A short calculation lets us replace $V^*$ in \e{vb1}--\e{vb3} with $U$ as follows:
\begin{align}
   \hat\rho_r  & =  2^{-(2r+2)/2} (2r+2)!! \left[ t^{r+1}\right] U(t), \label{vc1}\\
   \g_r & =  2^{-(2r+1)/2} (2r+1)!!  \left[ t^{r+1/2}\right] U(t), \label{vc2}\\
   \tau_r & =  2^{-(2r-1)/2} (2r-1)!!  \left[ t^{r-1/2}\right]U(t)'/(1-U(t)). \label{vc3}
\end{align}
Then \e{vc1} and \e{vc2} imply  \e{rg}. Since
\begin{equation*}
  U' = \frac{U}{U-1} = 1+ \frac{1}{U-1} \implies \frac{1}{U-1} = U'-1,
\end{equation*}
we can also replace $U'/(1-U)$ in \e{vc3} by $U'-(U')^2$. This yields \e{rgg}.
\end{proof}

\begin{prop} \label{top}
Theorem \ref{ky} is equivalent to
\begin{equation}\label{ccj}
  2(n-4)!! \sum_{\substack{x+y=n \\ x \text{ odd}}} x c_x \cdot y c_y =
  -n!! c_n+4(n-2)!! c_{n-2}+
  \sum_{\substack{x+y=n \\ x \text{ odd}}} (-1)^{y/2} x!! c_x \cdot y!! c_y,
\end{equation}
for all odd $n\gqs 3$.
\end{prop}
\begin{proof}
This follows by substituting the formulas of Proposition \ref{lus} into \e{alx}.
\end{proof}


At this point we have reduced things to a clear and seemingly elementary question: how can \e{crec} be used to prove \e{ccj}?
Note that  \e{ccj} may also be written more symmetrically as
\begin{equation*}
  \sum_{x+y=n} c_x c_y \left[(n-4)!! x y -\frac{\cos(\frac{\pi x}2)+\cos(\frac{\pi y}2)}2 x!! y!! \right] = -n!! c_n + 4(n-2)!! c_{n-2}.
\end{equation*}

\section{Stirling cycle numbers}

The Stirling cycle numbers $\stira{n}{k}$ may be defined recursively for integers $n$ and $k$ with $n\gqs 0$ by the initial condition $\stira{0}{k} = \delta_{k,0}$ and the relation
\begin{equation}\label{stcy}
  \stira{n+1}{k}=\stira{n}{k-1}+n \stira{n}{k}. 
\end{equation}
See \cite[Sect. 6.1]{Knu} for more information about them. We also require a variant $\stira{n}{k}^*$ that  satisfies \e{stcy} but has the different initial conditions
$$
\stira{0}{k}^*=0 \quad \text{for} \quad k\lqs 0,
\qquad
  \stira{n}{n}^* = \frac{n-1}n \quad \text{for} \quad n\gqs 1.
  $$
These numbers for $n\lqs 5$ are displayed in Tables \ref{figa} and \ref{figb} with the initial conditions highlighted and zeros omitted.
\SpecialCoor
\psset{griddots=5,subgriddiv=0,gridlabels=0pt}
\psset{xunit=0.5cm, yunit=0.5cm}
\psset{linewidth=1pt}
\psset{dotsize=2pt 0,dotstyle=*,arrowsize=2pt 3}
\begin{table}[ht]
\centering
\begin{pspicture}(-7,-0.5)(6,7.5) 

\newrgbcolor{pale}{0.98 0.83 0.79}
\newrgbcolor{pale2}{0.99 0.7 0.62}

\pspolygon[linecolor=pale,fillstyle=solid,fillcolor=pale]
(-5.5,4.5)(6.5,4.5)(6.5,5.5)(-5.5,5.5)

\rput(0,5){$1$}
\rput(1,4){$1$}
\rput(0,3){$1$}
\rput(2,3){$1$}
\rput(-1,2){$2$}
\rput(1,2){$3$}
\rput(3,2){$1$}
\rput(-2,1){$6$}
\rput(0,1){$11$}
\rput(2,1){$6$}
\rput(4,1){$1$}
\rput(-3,0){$24$}
\rput(-1,0){$50$}
\rput(1,0){$35$}
\rput(3,0){$10$}
\rput(5,0){$1$}

\rput(-7,0){$n=5$}
\rput(-7,1){$n=4$}
\rput(-7,2){$n=3$}
\rput(-7,3){$n=2$}
\rput(-7,4){$n=1$}
\rput(-7,5){$n=0$}

\rput{45}(2,7){$k=0$}
\rput{45}(3,6){$k=1$}
\rput{45}(4,5){$k=2$}
\rput{45}(5,4){$k=3$}
\rput{45}(6,3){$k=4$}
\rput{45}(7,2){$k=5$}

\end{pspicture}
\caption{The Stirling cycle numbers $\stira{n}{k}$}
\label{figa}
\end{table}
\SpecialCoor
\psset{griddots=5,subgriddiv=0,gridlabels=0pt}
\psset{xunit=0.5cm, yunit=0.5cm}
\psset{linewidth=1pt}
\psset{dotsize=2pt 0,dotstyle=*,arrowsize=2pt 3}
\begin{table}[ht]
\centering
\begin{pspicture}(-7,-0.5)(6,7.5) 
\newrgbcolor{pale}{0.98 0.83 0.79}
\newrgbcolor{pale2}{0.99 0.7 0.62}

\pspolygon[linecolor=pale,fillstyle=solid,fillcolor=pale]
(-5.5,4.5)(-0.2,4.5)(4.8,-0.5)(6.2,-0.5)(0.2,5.5)(-5.5,5.5)

\rput(2,3){$\frac 12$}
\rput(1,2){$1$}
\rput(3,2){$\frac 23$}
\rput(0,1){$3$}
\rput(2,1){$3$}
\rput(4,1){$\frac 34$}
\rput(-1,0){$12$}
\rput(1,0){$15$}
\rput(3,0){$6$}
\rput(5,0){$\frac 45$}

\rput(-7,0){$n=5$}
\rput(-7,1){$n=4$}
\rput(-7,2){$n=3$}
\rput(-7,3){$n=2$}
\rput(-7,4){$n=1$}
\rput(-7,5){$n=0$}

\rput{45}(2,7){$k=0$}
\rput{45}(3,6){$k=1$}
\rput{45}(4,5){$k=2$}
\rput{45}(5,4){$k=3$}
\rput{45}(6,3){$k=4$}
\rput{45}(7,2){$k=5$}

\end{pspicture}
\caption{The modified Stirling cycle numbers $\stira{n}{k}^*$}
\label{figb}
\end{table}
As with the binomial coefficients $\binom{n}{k}$, we have that $\stira{n}{k}$ is zero for $n \gqs 0$ if $k$ is outside the range $0\lqs k\lqs n$. When manipulating sums containing these terms it is usually simpler to have the index range over all integers $k$, though the sum is really finite. For $n \gqs 1$ we have $\stira{n}{k}^*=0$ for  $k\lqs 1$. As we see in Section \ref{more}, $\stira{n}{k}^*$ may be nonzero for $k>n$.

An initially mysterious aspect of \e{ccj} is the appearance of the double factorials. We will find that double factorials arise in the description of the sums
\begin{equation*}
  S_n(k):= \left[ t^{n/2}\right] U(t)^k = \sum_{j_1+j_2+ \cdots +j_k = n} c_{j_1}c_{j_2} \cdots c_{j_k}.
\end{equation*}

\begin{lemma}
For all integers $n$, $k\gqs 1$,
\begin{equation}\label{all}
  \frac{S_n(k+1)}{k+1}  = \frac{S_n(k)}{k} + \frac 2n S_{n-2}(k).
\end{equation}
\end{lemma}
\begin{proof}
Note that
\begin{align*}
  k \sum_{j_1+ \cdots +j_k = n} (j_1 c_{j_1}) c_{j_2} \cdots c_{j_k}
  & = \sum_{j_1+ \cdots +j_k = n} \Bigl\{ (j_1 c_{j_1}) c_{j_2} \cdots c_{j_k}+ \cdots +   c_{j_1} c_{j_2} \cdots (j_k c_{j_k})\Bigr\}\\
  & = n \sum_{j_1+ \cdots +j_k = n} c_{j_1}c_{j_2} \cdots c_{j_k},
\end{align*}
and therefore
\begin{equation}\label{the}
   \sum_{j_1+ \cdots +j_k = n} (j_1 c_{j_1}) c_{j_2} \cdots c_{j_k} = n\frac{S_n(k)}{k}.
\end{equation}
By \e{crec},
\begin{multline*}
  \sum_{j_1+ \cdots +j_{k+1} = n} (j_1 c_{j_1}c_{j_2}) c_{j_3}\cdots c_{j_{k+1}}
  = \sum_{r+j_3+ \cdots +j_{k+1} = n} (r c_r) c_{j_3}\cdots c_{j_{k+1}}\\
  + 2\sum_{r+j_3+ \cdots +j_{k+1} = n} ( c_{r-2}) c_{j_3}\cdots c_{j_{k+1}},
\end{multline*}
and hence, using \e{the},
\begin{equation*}
  n\frac{S_n(k+1)}{k+1}  = n\frac{S_n(k)}{k} + 2 S_{n-2}(k). \qedhere
\end{equation*}
\end{proof}

To simplify the following expressions, let $d_n:= n!! c_n$.

\begin{theorem} \label{sec}
 For all integers $n\gqs 0$ and $k\gqs 1$,
 \begin{equation} \label{!!}
  n!! \frac{S_n(k)}{k} = \sum_{j=0}^{k-1} 2^j \left(\stira{k}{k-j}  d_{n-2j}- \stira{k}{k-j}^* \delta_{n-2j,0}\right).
\end{equation}
\end{theorem}
\begin{proof}
We first verify \e{!!} when $n=0$. It reduces to
\begin{equation*}
  \frac{S_0(k)}{k} = \stira{k}{k} - \stira{k}{k}^*
\end{equation*}
and this is true as $S_0(k)=1$, $\stira{k}{k}=1$ and $\stira{k}{k}^*=(k-1)/k$.
To prove \e{!!} in general, use induction on $k$. The case $k=1$ is easily checked and the induction step when $n\gqs 1$ follows from \e{stcy} and \e{all}.
\end{proof}

Theorem \ref{sec} will allow us to give an expansion of $U'$, starting with \e{upx}, that provides useful information in the next proposition. A similar expansion of $(U')^2$ in \e{up} will give formulas for the left side of \e{ccj}.

\begin{prop}
For all integers $n \gqs 0$,
\begin{equation} \label{radu}
   \frac n2 \frac{d_n}{2^{n/2}}
  = \sum_{j=0}^{n/2} \frac{d_{n-2j}}{ 2^{n/2-j}} \alpha(j) -  \alpha^*(n/2),
\end{equation}
where $\alpha(j)$ and $\alpha^*(j)$ are defined in \e{aj} and \e{x2m}, and $d_n:= n!! c_n$.
\end{prop}
\begin{proof}
The relation $U=\exp(U-t-1)$ implies
\begin{equation*}
  U-t-1 = \log(U)=\log(1-(1-U))=\sum_{\ell\gqs 1} \frac{(-1)^{\ell-1}}{\ell}(U-1)^\ell.
\end{equation*}
Multiply through by $U'=U/(U-1)$ to find
\begin{equation} \label{upx}
  U't =  \sum_{\ell\gqs 2} \frac{(-1)^{\ell}}{\ell}U(U-1)^{\ell-1}
  =  \sum_{\ell\gqs 2} \frac{1}{\ell}\sum_{k=1}^{\ell}(-1)^{k}\binom{\ell-1}{k-1} U^{k}.
\end{equation}
Then apply $[t^{n/2}]$. As $(U-1)^{\ell-1} = (\sqrt{2} t^{1/2}+\cdots)^{\ell-1}$, it is apparent that we obtain nonzero contributions only when $\ell\lqs n+1$. For our following arguments it is convenient to take an upper limit of $n+2$:
\begin{equation} \label{up2x}
  \frac n2 c_n
  =  \sum_{\ell= 2}^{n+2} \frac{1}{\ell}\sum_{k=1}^{\ell}(-1)^{k}\binom{\ell-1}{k-1} S_n(k).
\end{equation}
Hence, by Theorem \ref{sec},
\begin{equation} \label{ukr}
  \frac n2 d_n
  =  \sum_{\ell= 2}^{n+2} \frac{1}{\ell}\sum_{k=1}^{\ell}(-1)^{k}\binom{\ell-1}{k-1}
  k\left\{\sum_{j=0}^{k-1} 2^j \stira{k}{k-j}  d_{n-2j}- \sum_{j=0}^{k-1} 2^j \stira{k}{k-j}^* \delta_{n,2j} \right\}.
\end{equation}
Interchanging summations finds that \e{ukr} equals
\begin{multline} \label{ukrq}
    \sum_{j=0}^{n+1} 2^j d_{n-2j} \sum_{\ell= 2}^{n+2} \frac{1}{\ell}\sum_{k=1}^\ell (-1)^{k} k \stira{k}{k-j}\binom{\ell-1}{k-1}\\
  -\sum_{j=0}^{n+1} 2^j \delta_{n,2j} \sum_{\ell= 2}^{n+2} \frac{1}{\ell}\sum_{k=1}^\ell (-1)^{k} k \stira{k}{k-j}^*\binom{\ell-1}{k-1}.
\end{multline}
We must have $j\lqs n/2$ in both sums in \e{ukrq} to have nonzero contributions.
Label the components  as
\begin{align}
  \alpha(j) & :=  \sum_{\ell= 2}^{2j+2} \frac{1}{\ell}\sum_{k=1}^\ell (-1)^{k} k \stira{k}{k-j}\binom{\ell-1}{k-1}, \label{aj}\\
  \alpha^*(j) & :=  \sum_{\ell= 2}^{2j+2} \frac{1}{\ell}\sum_{k=1}^\ell (-1)^{k} k \stira{k}{k-j}^* \binom{\ell-1}{k-1}, \label{x2m}
\end{align}
with $\alpha^*(n/2) := 0$ for $n$ odd. We will see shortly in Proposition \ref{pro} that the inner sums  in \e{aj} and \e{x2m} can only be nonzero for $\ell \lqs 2j+2$,  and this is why the upper limit of summation is $2j+2$ instead of $n+2$. Then
\e{radu} follows.
\end{proof}

With \e{rg} and $n=2m+1$ we obtain
\begin{equation} \label{gmj}
  m \g_m = \sum_{j=0}^{m-1} \alpha(j+1) \g_{m-1-j} \qquad (m\gqs 0),
\end{equation}
from \e{radu} since $\alpha(0)=1/2$.
For $n=2m+2$ we obtain
\begin{equation} \label{rhj}
  (m+1/2) \hat\rho_m = \alpha^*(m+1)-\alpha(m+1)+\sum_{j=0}^{m-1} \alpha(j+1) \hat\rho_{m-1-j} \qquad (m\gqs 0).
\end{equation}

\section{Simplifying $\alpha(j)$ and $\alpha^*(j)$}
For integers $n$ and $k$, we will need a variant $\eud{n}{k}^*$ of the usual second-order Eulerian numbers $\eud{n}{k}$. They  satisfy \e{eur}, but with the initial condition $\eud{1}{k}^*=\delta_{k,-1}$.
These second-order Eulerian numbers are shown in Tables \ref{figc} and \ref{figd} with the initial conditions highlighted and zeros omitted, as before.
\SpecialCoor
\psset{griddots=5,subgriddiv=0,gridlabels=0pt}
\psset{xunit=0.5cm, yunit=0.5cm}
\psset{linewidth=1pt}
\psset{dotsize=2pt 0,dotstyle=*,arrowsize=2pt 3}
\begin{table}[ht]
\centering
\begin{pspicture}(-8,-0.5)(6,7.5) 
\newrgbcolor{pale}{0.98 0.83 0.79}
\newrgbcolor{pale2}{0.99 0.7 0.62}

\pspolygon[linecolor=pale,fillstyle=solid,fillcolor=pale]
(-6.5,4.5)(5,4.5)(5,5.5)(-6.5,5.5)

\rput(0,5){$1$}
\rput(-1,4){$1$}
\rput(0,3){$2$}
\rput(-2,3){$1$}
\rput(-1,2){$8$}
\rput(1,2){$6$}
\rput(-3,2){$1$}
\rput(-2,1){$22$}
\rput(0,1){$58$}
\rput(2,1){$24$}
\rput(-4,1){$1$}
\rput(-3,0){$52$}
\rput(-1,0){$328$}
\rput(1,0){$444$}
\rput(3,0){$120$}
\rput(-5,0){$1$}

\rput(-8,0){$n=5$}
\rput(-8,1){$n=4$}
\rput(-8,2){$n=3$}
\rput(-8,3){$n=2$}
\rput(-8,4){$n=1$}
\rput(-8,5){$n=0$}

\rput{45}(2,7){$k=0$}
\rput{45}(3,6){$k=1$}
\rput{45}(4,5){$k=2$}
\rput{45}(5,4){$k=3$}
\rput{45}(6,3){$k=4$}

\end{pspicture}
\caption{The second-order Eulerian numbers $\eud{n}{k}$}
\label{figc}
\end{table}
\SpecialCoor
\psset{griddots=5,subgriddiv=0,gridlabels=0pt}
\psset{xunit=0.5cm, yunit=0.5cm}
\psset{linewidth=1pt}
\psset{dotsize=2pt 0,dotstyle=*,arrowsize=2pt 3}
\begin{table}[ht]
\centering
\begin{pspicture}(-8,-0.5)(6,6) 
\newrgbcolor{pale}{0.98 0.83 0.79}
\newrgbcolor{pale2}{0.99 0.7 0.62}

\pspolygon[linecolor=pale,fillstyle=solid,fillcolor=pale]
(-6.5,3.5)(5,3.5)(5,4.5)(-6.5,4.5)

\rput(-3,4){$1$}
\rput(-2,3){$3$}
\rput(-1,2){$12$}
\rput(-3,2){$3$}
\rput(-2,1){$42$}
\rput(0,1){$60$}
\rput(-4,1){$3$}
\rput(-3,0){$108$}
\rput(-1,0){$474$}
\rput(1,0){$360$}
\rput(-5,0){$3$}

\rput(-8,0){$n=5$}
\rput(-8,1){$n=4$}
\rput(-8,2){$n=3$}
\rput(-8,3){$n=2$}
\rput(-8,4){$n=1$}
\rput(-8,5){$n=0$}

\rput{45}(0.5,5.5){$k=0$}
\rput{45}(1.5,4.5){$k=1$}
\rput{45}(2.5,3.5){$k=2$}
\rput{45}(3.5,2.5){$k=3$}

\end{pspicture}
\caption{The modified second-order Eulerian numbers $\eud{n}{k}^*$}
\label{figd}
\end{table}
For $n\gqs 0$ we have $\eud{n}{k}=0$ outside $0\lqs k \lqs n$ and $\eud{n}{k}^*=0$ outside $-1\lqs k \lqs n-2$.
The key relations they satisfy are
\begin{alignat}{2}
  \stira{m}{m-n} & = \sum_k \eu{n}{k} \binom{m+k}{2n} \quad &\qquad &(m,n \in \Z_{\gqs 0}), \label{ir} \\
  \stira{m}{m-n}^* & = \sum_k \eu{n}{k}^* \binom{m+k}{2n} \quad &\qquad &(m,n \in \Z_{\gqs 1}). \label{ir2}
\end{alignat}
The identity \e{ir} is \cite[Eq. (6.44)]{Knu}, and is  proved by induction on $n$.  A similar proof gives \e{ir2}, where the induction starts with the easily verified identity
\begin{equation*}
  \stira{m}{m-1}^* = \binom{m-1}{2} \qquad (m \in \Z_{\gqs 1}).
\end{equation*}

\begin{prop} \label{pro}
We have $\alpha(0)=\alpha^*(0)=1/2$ and for integers $j  \gqs 1$
\begin{align}
  \alpha(j) & = \frac{-1}{2(j+1)} \sum_{r= 0}^{j} (-1)^r \eu{j+1}{r}\binom{2j+1}{r+1}^{-1}, \label{po}\\
  \alpha^*(j) & = \frac{-1}{2(j+1)} \sum_{r= 0}^{j} (-1)^r \eu{j+1}{r}^*\binom{2j+1}{r+1}^{-1}. \label{po2}
\end{align}
\end{prop}
\begin{proof}
For the inner sum in \e{aj} use
\begin{equation*}
  k \stira{k}{k-j} =  \stira{k+1}{k-j} -  \stira{k}{k-j-1}
\end{equation*}
 and recombine to show, for integers $\ell$ and $j$ with $\ell\gqs 1$,
\begin{equation} \label{job}
  \sum_{k} (-1)^{k} k \stira{k}{k-j}\binom{\ell-1}{k-1} =  \sum_{k} (-1)^{k}  \stira{k+1}{k-j}\binom{\ell}{k}.
\end{equation}
Next use \e{ir} and change the order of summation to see that, when $j\gqs -1$, \e{job} equals
\begin{equation*}
  \sum_r \eu{j+1}{r}
  \sum_{k} (-1)^{k} \binom{k+r+1}{2j+2}\binom{\ell}{k}
  =
  \sum_r \eu{j+1}{r} (-1)^\ell \binom{r+1}{2j+2-\ell},
\end{equation*}
where we used the identity \cite[Eq. (5.24)]{Knu}. It is now clear that the inner sum in \e{aj} is zero for $\ell > 2j+2$, as  claimed earlier. Therefore
\begin{equation} \label{op}
  \alpha(j) =
  \sum_r \eu{j+1}{r}
   \sum_{\ell\gqs 2} \frac{(-1)^{\ell}}{\ell}  \binom{r+1}{2j+2-\ell} \qquad (j\gqs 0).
\end{equation}
Note that in \e{op} we need $r\lqs j$ to have $\eud{j+1}{r} \neq 0$ and
 $2j+2-\ell \lqs r+1\lqs j+1$ to have $\binom{r+1}{2j+2-\ell}\neq 0$. Therefore we may assume that $\ell \gqs j+1$. This agrees with the condition  $\ell\gqs 2$ in \e{op} provided that $j\gqs 1$. Hence
 \begin{equation} \label{opu}
  \alpha(j) =
  \sum_{r=0}^j \eu{j+1}{r}
   \sum_{\ell=j+1}^{2j+2} \frac{(-1)^{\ell}}{\ell}  \binom{r+1}{\ell+r-1-2j} \qquad (j\gqs 1).
\end{equation}
One last required identity, from \cite[Eqns. (5.40), (5.41)]{Knu}, is
\begin{equation*}
  \sum_{m} \frac{(-1)^m}{x+m} \binom{n}{m}= \frac{n!}{x(x+1) \cdots (x+n)}.
\end{equation*}
 Applying this to the inner sum in \e{opu} yields \e{po}.

The same steps  prove \e{po2}, using \e{ir2} instead of \e{ir}.
\end{proof}

We next express $\alpha(j)$ and $\alpha^*(j)$ as integrals. To do this, define
\begin{equation*}
  E_n(x):= \sum_k \eu{n}{k} x^k, \qquad F_n(x):=\frac{E_n(x)}{(x-1)^{2n}}.
\end{equation*}
Then it may be seen that the equalities $F_0(x)=1$ and
\begin{equation} \label{rfn}
   \left( \frac x{1-x} F_n(x)\right)'=F_{n+1}(x)
\end{equation}
conveniently encode the recursive definition of the second-order Eulerian numbers $\eud{n}{k}$. For example,
\begin{equation*}
  F_1(x)=\frac 1{(1-x)^2}, \qquad F_2(x)=\frac{1+2x}{(1-x)^4}, \qquad F_3(x)=\frac{1+8x+6x^2}{(1-x)^6}.
\end{equation*}

\begin{prop} \label{val}
For integers $n$, $a$, $b$ with $n\gqs 1$, $a \gqs 0$ and $n+b+1\gqs 0$,
\begin{equation} \label{khe}
  \int_{-\infty}^0 t^a (1-t)^{-a-b-2} F_n(t) \, dt = \frac{(-1)^a}{2n+a+b+1}
  \sum_{k= 0}^{n-1} (-1)^k \eu{n}{k}\binom{2n+a+b}{k+a}^{-1}.
\end{equation}
\end{prop}
\begin{proof}
Starting with
\begin{equation} \label{pra}
  \int_0^1 x^a (1-x)^b F_n\left( \frac x{x-1}\right) \, dx =  \sum_{k= 0}^{n-1} (-1)^k \eu{n}{k}
  \int_0^1 x^{a+k} (1-x)^{b-k+2n} \, dx,
\end{equation}
and evaluating the beta integral shows \e{pra} equals $(-1)^a$ times the right of \e{khe}. The change of variables $t=x/(x-1)$ on the left of \e{pra} completes the proof.
\end{proof}

With this proposition, \e{po} becomes
\begin{equation} \label{zy}
  \alpha(n)= \int_{-\infty}^0 t (1-t)^{-1} F_{n+1}(t) \, dt = \int_{-\infty}^0 t (1-t)^{-1} F_0(t) F_{n+1}(t) \, dt,
\end{equation}
for $n\gqs 1$. As in Majer's work in \cite{mo}, integrating by parts using \e{rfn} shows
\begin{equation} \label{ann}
  \alpha(n)= (-1)^j \int_{-\infty}^0 t (1-t)^{-1} F_j(t) F_{n+1-j}(t) \, dt,
\end{equation}
for $n\gqs 1$ and $0\lqs j \lqs n+1$.

To find similar formulas for $\alpha^*(n)$, define
\begin{equation*}
  E^*_n(x):= \sum_k \eu{n}{k}^* x^k, \qquad F^*_n(x):=\frac{E^*_n(x)}{(x-1)^{2n}}, \qquad (n\gqs 1).
\end{equation*}
 Then $F^*_n(x)$ satisfies \e{rfn} but with the initial condition $F^*_1(x)=1/(x(x-1)^2)$ since $E^*_1(x)=1/x$.

\begin{prop}
For integers $n\gqs 1$,
\begin{align}\label{xt}
 \alpha^*(n) & = \int_{-\infty}^0 t (1-t)^{-1} F^*_{n+1}(t) \, dt \\
  & = (-1)^n \int_{-\infty}^0  (1-t)^{-3} F_{n}(t) \, dt \label{xt2}\\
  & = (-1)^n \frac{1}{2(n+1)} \sum_{k= 0}^{n-1} (-1)^k \eu{n}{k}\binom{2n+1}{k}^{-1}. \label{xt3}
\end{align}
\end{prop}
\begin{proof}
The equality \e{xt} follows by the same reasoning as \e{zy}, with Proposition \ref{val} valid when $F_n(t)$ and $\eud{n}{k}$ are replaced by their starred versions provided $n\gqs 2$. Including $F_0(t)$ in \e{xt} and integrating by parts gives
\begin{equation*}
  \alpha^*(n) = (-1)^j \int_{-\infty}^0 t (1-t)^{-1} F_j(t) F^*_{n+1-j}(t) \, dt,
\end{equation*}
for  $0\lqs j \lqs n$. Then \e{xt2} is the case $j=n$ and \e{xt3} follows from \e{xt2} by Proposition \ref{val}.
\end{proof}

\section{Bernoulli numbers} \label{ber}
Let $B_n$ be the $n$th Bernoulli number, using the convention $B_1=1/2$. In this section we exploit the connection to these numbers given by the asymptotic expansion of $\log \G(n)$ we saw in \e{logg} in Section \ref{ini}. The next result is suggested by comparing \e{hy} and \e{gmj}.

\begin{prop} \label{apk}
For $j\gqs 0$,
\begin{equation} \label{abba}
  \alpha(j) = B_{j+1}/(j+1). 
\end{equation}
\end{prop}
\begin{proof}
Use \e{iff} with \e{gmj} to show
\begin{equation} \label{pop}
  \sum_{j=0}^{r} (-1)^{j} (j+1) \g_{j+1}\g_{r-j} =  \alpha(r+1),
\end{equation}
for $r\gqs 0$. Then \e{iff2} and \e{pop} imply \e{abba} for $j\gqs 1$. Also  $\alpha(0)=1/2 = B_1$.
\end{proof}

It follows now from  \e{rhj} that, for integers $m\gqs 0$,
\begin{equation}
  (m+1/2) \hat\rho_m  = \alpha^*(m+1)- \frac{B_{m+2}}{m+2}+\sum_{j=0}^{m-1} \frac{B_{j+2}}{j+2} \hat\rho_{m-1-j}.\label{rhj2}
\end{equation}

Equation \e{move} may also be proved next.

\begin{theorem} \label{mov}
For all integers $n\gqs 0$,
\begin{equation} \label{nkb}
  \sum_{k= 0}^n (-1)^k \eu{n}{k}\binom{2n+1}{k+1}^{-1} = 2B_{n+1}.
\end{equation}
\end{theorem}
\begin{proof}
Equality \e{nkb} is true for $n=0$.  When $n \gqs 1$, by \e{ann} for $j=1$ and \e{abba},
\begin{align}
  \frac{B_{n+1}}{n+1} & = - \int_{-\infty}^0 t (1-t)^{-1} F_1(t) F_{n}(t) \, dt \notag\\
  & = - \int_{-\infty}^0 t (1-t)^{-3} F_{n}(t) \, dt. \label{abc}
\end{align}
Proposition \ref{val} with $a=1$ and $b=0$ finishes the argument.
\end{proof}

Another proof of Theorem \ref{mov} is provided in \cite{Fu}.  For more related  identities see \cite{mo,RU19,Fu}.

\section{Proof of the main theorem}
\begin{prop} \label{ert}
Let $n\gqs 3$ be an odd integer. Then  \e{ccj} is true   if and only if
\begin{equation}\label{ccj2}
  \sum_{j=0}^{(n-3)/2} \beta(j) 2^j d_{n-2-2j} = \frac{n-2}2
  \Bigg(2d_{n-2}- \frac{d_n}2+ \sum_{\substack{x+y=n \\ x \text{ odd}}} \frac{(-1)^{y/2}}2 d_x d_y
  \Bigg),
\end{equation}
for $\beta(j)$ defined in \e{bj} and $d_n$ our shorthand for $n!! c_n$.
\end{prop}
\begin{proof}
Multiply \e{upx} by $U'=U/(U-1)$ to find
\begin{equation} \label{up}
  (U')^2 t
  =  \sum_{\ell\gqs 2} \frac{1}{\ell}\sum_{k=2}^{\ell}(-1)^{k}\binom{\ell-2}{k-2} U^{k},
\end{equation}
and applying $[t^{n/2}]$ for $n$ in $\Z_{\gqs 0}$ then shows
\begin{equation*}
  \frac 14 \sum_{x+y=n+2} x c_x \cdot y c_y
  = \sum_{\ell= 2}^{n+3} \frac{1}{\ell}\sum_{k=2}^{\ell}(-1)^{k}\binom{\ell-2}{k-2} S_n(k).
\end{equation*}
Copying the treatment of the right side of \e{up2x} requires, for integers $j \gqs 0$,
\begin{align}
  \beta(j) & :=  \sum_{\ell= 2}^{2j+3} \frac{1}{\ell}\sum_{k=2}^\ell (-1)^{k} k \stira{k}{k-j}\binom{\ell-2}{k-2}, \label{bj}\\
  \beta^*(j) & :=  \sum_{\ell= 2}^{2j+3} \frac{1}{\ell}\sum_{k=2}^\ell (-1)^{k} k \stira{k}{k-j}^* \binom{\ell-2}{k-2}, \label{y2m}
\end{align}
with $\beta^*(n/2) := 0$ for $n$ odd, and we find
\begin{equation} \label{mad}
   \frac {n!!}4 \sum_{x+y=n+2} x c_x \cdot y c_y
  = \sum_{j=0}^{n/2} 2^j d_{n-2j}  \beta(j) - 2^{n/2} \beta^*(n/2).
\end{equation}
Solving for the left sides of \e{ccj} and \e{mad} completes the proof.
\end{proof}

\begin{prop} \label{top2}
The equality \e{ccj2} is true for all odd $n\gqs 3$  if and only if
\begin{equation} \label{top3}
  (m+1/2) \hat\rho_m = (-1)^m \beta(m) +2\frac{B_{m+1}}{m+1}+\sum_{j=0}^{m-1} \frac{B_{j+2}}{j+2} \hat\rho_{m-1-j}, \qquad \text{for all $m\gqs 1$.}
\end{equation}
\end{prop}
\begin{proof}
We will use the following generating functions,
\begin{equation*}
  A(x):=\sum_{j\gqs 0} \beta(j) 2^j x^j, \quad C(x):=\sum_{j\gqs 0} d_{2j+1}  x^j, \quad D(x):=\sum_{j\gqs 1}
  \left( \frac{(-1)^j}2 d_{2j} +2 \delta_{j,1}\right) x^j.
\end{equation*}
Then it is straightforward to check that \e{ccj2} being true for all odd $n\gqs 3$ is equivalent to
\begin{equation*}
  x^{-1/2} C(x) A(x) = \left( x^{-1/2} C(x) D(x)\right)',
\end{equation*}
which we may rewrite as
\begin{equation} \label{ppp}
  A(x)=- \frac{D(x)}{2x}+D'(x)+\frac{C'(x)}{C(x)} D(x).
\end{equation}
Recalling \e{rg} and $g(x)$ in \e{gz}, we have
\begin{equation*}
  C(x)=\sum_{j\gqs 0} \g_j 2^{j+1/2} x^j = \sqrt{2} g(2x).
\end{equation*}
Hence, by \e{hx},
\begin{equation*}
  \frac{C'(x)}{C(x)} = 2\frac{g'(2x)}{g(2x)} = \sum_{j\gqs 0}  2^{j+1} \frac{B_{j+2}}{j+2} x^j.
\end{equation*}
Using this in \e{ppp} and comparing powers of $x$ completes the proof.
\end{proof}

Following the steps in the proof of Proposition \ref{pro} lets us express $\beta(j)$ and $\beta^*(j)$ similarly:

\begin{prop} \label{pro2}
We have $\beta(0)=2/3$, $\beta^*(0)=1/6$ and for integers $j \gqs 1$
\begin{align}
  \beta(j) & = \frac{-1}{2j+3} \sum_{r= 0}^j (-1)^r \eu{j+1}{r}\binom{2j+2}{r+2}^{-1}, \label{pb}\\
  \beta^*(j) & = \frac{-1}{2j+3} \sum_{r= 0}^j (-1)^r \eu{j+1}{r}^*\binom{2j+2}{r+2}^{-1}. \label{pb2}
\end{align}
\end{prop}

Applying Proposition \ref{val} to \e{pb} also gives
\begin{equation}\label{rut}
  \beta(m)= -\int_{-\infty}^0 t^2 (1-t)^{-2} F_{m+1}(t) \, dt \qquad (m\gqs 1).
\end{equation}

\begin{proof}[Proof of Theorem \ref{ky}]
By Propositions \ref{top}, \ref{ert} and \ref{top2}, we must prove \e{top3}. Subtracting \e{rhj2} from it gives the goal
\begin{equation} \label{goa}
   \alpha^*(m+1)
   =(-1)^m \beta(m) +\frac{B_{m+2}}{m+2} +2\frac{B_{m+1}}{m+1} \qquad (m \gqs 1),
\end{equation}
which no longer involves the sequences $\rho_j$ and $\g_j$. Combine the identity
\begin{equation*}
  \frac 1{(1-t)^3} = \frac{t^2}{(1-t)^2} + \frac t{(1-t)^3} + \frac{2t}{1-t} +1
\end{equation*}
 with the integrals
\begin{align*}
  \int_{-\infty}^0  (1-t)^{-3} F_{m+1}(t) \, dt & = (-1)^{m+1} \alpha^*(m+1), \\
  \int_{-\infty}^0 t^2 (1-t)^{-2} F_{m+1}(t) \, dt & = -\beta(m),\\
  \int_{-\infty}^0 t (1-t)^{-3} F_{m+1}(t) \, dt & = - \frac{B_{m+2}}{m+2}, \\
  \int_{-\infty}^0 t (1-t)^{-1} F_{m+1}(t) \, dt & = \frac{B_{m+1}}{m+1},
\end{align*}
which are \e{xt2}, \e{rut}, \e{abc} and \e{zy} respectively, and
\begin{equation*}
  \int_{-\infty}^0  F_{m+1}(t) \, dt = \left. \frac t{1-t}F_m(t) \right]_{-\infty}^0 =0.
\end{equation*}
Then  \e{goa} follows.
\end{proof}

This completes the proof of Theorem \ref{ky}, showing that the asymptotic expansion coefficients of $\theta_n$ and $\Psi_n$ agree up to an alternating sign. The initial version of this paper asked if there could be a shorter proof, since \e{rb} and \e{ei} are so similar, and it is gratifying that  this challenge has already been taken up. It is also natural to  ask if there are any similar examples, linking a Taylor expansion like \e{rb} to a sum of its reciprocals.

\section{Properties of $\rho_r$ and $\g_r$} \label{proo}
We gather known formulas and recurrences for Ramanujan's coefficients $\rho_r$ and Stirling's coefficients $\g_r$ in this  section, and display the new relations we have uncovered.
As discussed in \cite{odm}, the De Moivre polynomials  $\dm_{n,k}(a_1, a_2, a_3, \dots)$ from \e{bell} are zero when $n<k$, and for $n\gqs k$ then they have the explicit form
 \begin{equation} \label{bell2}
  \dm_{n,k}(a_1, a_2, a_3, \dots) = \sum_{\substack{1j_1+2 j_2+ \dots +mj_m= n \\ j_1+ j_2+ \dots +j_m= k}}
 \binom{k}{j_1 , j_2 ,  \dots , j_m} a_1^{j_1} a_2^{j_2}  \cdots a_m^{j_m}.
\end{equation}
Here $m=n-k+1$ and the sum   is over all possible $j_1$, $j_2$,  \dots , $j_m \in \Z_{\gqs 0}$.

\subsection{Formulas}
The first closed formula for $\g_r$ seems to be due to Perron in \cite[p. 210]{Pe17}. He stated it in terms of coefficients of power series and it is equivalent to
\begin{equation}\label{gmy}
  \g_r = \sum_{k=0}^{2r}  \frac{(2r+2k-1)!!}{(-1)^k k!} \dm_{2r,k}\left( \frac{1}{3}, \frac{1}{4}, \frac{1}{5}, \dots \right),
\end{equation}
which is really the same as \e{vbw2}. This result is also equivalent to \cite[Thm. 2.7]{BM11}. Brassesco and M\'endez gave another formulation in \cite[Thm. 2.1]{BM11}, essentially derived by a change of variables $z \to e^z$, 
\begin{equation}\label{gmy2}
  \g_r = \sum_{k=0}^{2r}  \frac{(2r+2k-1)!!}{(-1)^k k!} \dm_{2r,k}\left( \frac{1}{3!}, \frac{1}{4!}, \frac{1}{5!}, \dots \right).
\end{equation}
The corresponding formulas for $\rho_r$ are shown in \cite[Eqns. (7.1), (7.2)]{OSra}:
\begin{align} \label{rfb}
   \rho_r   = \delta_{r,0} +{} & \sum_{k=0}^{2r+1}  \frac{(2r+2k)!!}{(-1)^k k!} \dm_{2r+1,k}\left( \frac{1}{3}, \frac{1}{4}, \frac{1}{5}, \dots \right) \\
    = - & \sum_{k=0}^{2r+1}  \frac{(2r+2k)!!}{(-1)^k k!} \dm_{2r+1,k}\left( \frac{1}{3!}, \frac{1}{4!}, \frac{1}{5!}, \dots \right). \label{rfb2}
\end{align}

These results \e{gmy} -- \e{rfb2} may be regarded as the basic expressions for $\g_r$ and $\rho_r$. Variations are possible with further  expansions for the De Moivre polynomials  as in Nemes \cite{nem} and  \cite[Sect. 7]{OSra}. The simplest variant uses $r$-associated Stirling numbers.
For integers $n$, $k$ and $r$   with $n$, $k\gqs 0$, $r\gqs 1$, write $\stirb{n}{k}_{\gqs r}$ for the number of ways to partition  $n$ elements into $k$  subsets, each with at least $r$ members. Let $\stira{n}{k}_{\gqs r}$ denote the number of ways to arrange $n$ elements into $k$ cycles, each of length at least $r$. We also set $\stirb{0}{k}_{\gqs r}=\stirb{0}{k}_{\gqs r}=\delta_{k,0}$. These generalize the usual Stirling numbers in the $r=1$ case.
The following formulas are due to Comtet \cite[p. 267]{Comtet} and  Brassesco and M\'endez  \cite[Thm. 2.4]{BM11} respectively,
\begin{equation}\label{ass}
  \g_r  = \sum_{k=0}^{2r} \frac{(-1)^k}{(2r+2k)!!} \stira{2r+2k}{k}_{\!\gqs 3}
  , \qquad
  \g_r  = \sum_{k=0}^{2r} \frac{(-1)^k}{(2r+2k)!!} \stirb{2r+2k}{k}_{\!\gqs 3}.
\end{equation}
Analogously, from \cite[Eqns. (7.26), (7.27)]{OSra},
\begin{align}\label{ass2}
   \rho_r = \delta_{r,0}  +{} & \sum_{k=0}^{2r+1} \frac{(-1)^k}{(2r+2k+1)!!} \stira{2r+2k+1}{k}_{\!\gqs 3},\\
     = - & \sum_{k=0}^{2r+1} \frac{(-1)^k}{(2r+2k+1)!!} \stirb{2r+2k+1}{k}_{\!\gqs 3}. \label{ass3}
\end{align}
See \cite{OSra} for proofs of all of the identities \e{gmy} -- \e{ass3} as well as generalizations.

Exponentiating \e{logg}  produces another formula for $\g_r$, as 
in \cite[Eq. (7.13)]{odm},
\begin{equation} \label{stx2}
  \g_r =  \sum_{k=0}^{r}  \frac{1}{k!} \dm_{r,k}\left(\frac{B_2}{2\cdot 1},0,\frac{B_4}{4\cdot 3},0, \cdots \right).
\end{equation}
There are hints from \e{hatt} below that there could be a similar identity for $\rho_r$.

\subsection{Recurrences and relations}
We have seen in \e{hy} the simple recurrence
\begin{equation} \label{rr}
  \g_0=1, \qquad r \g_r  = \sum_{j=0}^{r-1} \frac{B_{j+2}}{j+2} \g_{r-1-j} \qquad (r\gqs 1).
\end{equation}
This was used by Wrench in
\cite{Wr68} to tabulate values of $\g_r$. The corresponding result for $\rho_r$ appears in \e{rhj2} and it is convenient to state it in terms of
 $\hat\rho_r := \rho_r - \delta_{r,0}$:
\begin{equation} \label{hatt}
  (r+1/2) \hat\rho_r  = \alpha^*(r+1)- \frac{B_{r+2}}{r+2}+\sum_{j=0}^{r-1} \frac{B_{j+2}}{j+2} \hat\rho_{r-1-j} \qquad (r\gqs 0).
\end{equation}
The numbers $\alpha^*(n)$ are given in \e{xt3}, for example. Perhaps they have a simpler Bernoulli number-like formulation.

The coefficients $c_n$ give another route to obtaining $\rho_r$ and $\g_r$. Their definition in \e{crec} is equivalent to the recursion
\begin{equation}\label{erop}
  c_0=1, \quad c_1=2^{1/2}, \quad c_{n+1}=\frac{2^{1/2}}{n+2} c_n -\frac{2^{1/2}}{4} \sum_{j=2}^n c_j c_{n+2-j} \quad (n\gqs 1).
\end{equation}
Then as in \e{rg},
\begin{equation} \label{ink}
  \hat \rho_r = - \G(r+2) c_{2r+2}, \qquad \g_r = 2^{1/2} \frac{\G(r+\tfrac 32)}{\G(\tfrac 12)} c_{2r+1}.
\end{equation}
This method   is given by  Marsaglia in \cite{Mar86} for calculating $\rho_r$ and given by Marsaglia and Marsaglia in \cite{mar2} for calculating  $\g_r$. (They use the similar coefficients of $V^*(-x)=(U(x^2/2)-1)/x$ in \e{vs} instead of the coefficients $c_n$ of $U(t)$.)

Define the following convenient extension of the binomial coefficient $\binom{n}{m}$ for some non-integral $m$ values. With integers $0\lqs k \lqs n$:
$$
\binom{n}{k-\tfrac 12}:= \frac{\G(n+1)\G(\tfrac 12)^2}{\G(k+\tfrac 12)\G(n-k+\tfrac 12)} = \frac{(2n)!!}{(2k-1)!! (2n-2k-1)!!}.
$$
Then substituting \e{ink} in \e{crec} for odd $n=2r+3$ and even $n=2r+2$, provides the relations
\begin{gather}
  \sum_{j=0}^r \binom{r+\tfrac 32}{j+1}\hat \rho_j \g_{r-j}   = -\g_r,  \label{rr2}\\
  \sum_{j=0}^{r-1} \binom{r+1}{j+1}\hat \rho_j \hat \rho_{r-1-j} + \frac 1{2(r+2)}\sum_{j=0}^r \binom{r+2}{j+\tfrac 12} \g_j \g_{r-j}  =-2\hat \rho_{r-1}, \label{rr3}
\end{gather}
respectively, for all $r\gqs 0$.

With \e{ink}, our main theorem is equivalent to
\begin{equation} \label{our}
  \sum_{j=0}^r \left(\binom{r-\tfrac 12}{j}+\frac{(-1)^j}2 \right) \hat \rho_j \g_{r-j}   = -\g_r \qquad (r\gqs 0),
\end{equation}
by Proposition \ref{top} with odd $n=2r+3$. For an even analog of \e{our}, we have by \e{mad} with $n=2r$,
\begin{equation}\label{madder}
  r\sum_{j=0}^{r-1} \binom{r-1}{j} \hat \rho_j \hat \rho_{r-1-j}
  +\frac 12 \sum_{j=0}^{r} \binom{r}{j-\tfrac 12} \g_j \g_{r-j}
  = \beta(r) -\beta^*(r)- \sum_{j=0}^{r-1} \beta(j) \hat \rho_{r-1-j},
\end{equation}
for all $r\gqs 0$ with the numbers $\beta(j)$ and $\beta^*(j)$ given in Proposition \ref{pro2}.

\section{More on the modified Stirling cycle numbers} \label{more}
Using the recursive relation \e{stcy},  the range of
$\stira{n}{k}^*$ in Table \ref{figb} may be extended above and to the right of the initial conditions to find
\begin{equation*}
  \stira{0}{1}^* = \frac{1}{2}, \qquad
  \stira{0}{2}^* = -\frac{5}{12}, \qquad
  \stira{0}{3}^* = \frac{7}{18}, \qquad
  \stira{0}{4}^* = -\frac{1631}{4320}, \qquad \dots.
\end{equation*}
A clear way to compute these numbers is as follows. On row $n=0$ start with the sequence $1$, $\frac 12$, $\frac 13$, $\frac 14, \dots$. On the next row, with $n=1$, make the sequence $(1-\frac 12)/1$, $(\frac 12-\frac 13)/2$,
$(\frac 13-\frac 14)/3$, and so on. For the next row, repeat the procedure by taking differences and dividing by $1$, $2$, $3,$ $\dots$ again. In general, with the numbers in row $n$ labelled $a_{n,m}$ for $m\gqs 0$, we have $a_{0,m}=1/(m+1)$ and
\begin{equation}\label{tan}
  a_{n+1,m}=(a_{n,m}-a_{n,m+1})/(m+1) \qquad (n, m\gqs 0).
\end{equation}
If we define $\omega_n:=a_{n,0}$, the number at the start of the $n$th row,    then it is straightforward to verify that
\begin{equation*}
   \stira{0}{k}^* = (-1)^{k-1} \omega_k \qquad (k\gqs 1).
\end{equation*}
Table \ref{fige}  displays the initial part of the array $a_{n,m}$.
\SpecialCoor
\psset{griddots=5,subgriddiv=0,gridlabels=0pt}
\psset{xunit=0.6cm, yunit=0.6cm}
\psset{linewidth=1pt}
\psset{dotsize=2pt 0,dotstyle=*,arrowsize=2pt 3}
\begin{table}[ht]
\centering
\begin{pspicture}(0,0.5)(10,6) 
\newrgbcolor{pale}{0.98 0.83 0.79}
\newrgbcolor{pale2}{0.99 0.7 0.62}
\pspolygon[linecolor=pale2,fillstyle=solid,fillcolor=pale2]
(-0.2,4.5)(3.8,0.5)(5.2,0.5)(0.2,5.5)

\pspolygon[linecolor=pale,fillstyle=solid,fillcolor=pale]
(-0.5,4.5)(10.5,4.5)(10.5,5.5)(-0.5,5.5)

\rput(0,5){$1$}
\rput(2,5){$\frac 12$}
\rput(4,5){$\frac 13$}
\rput(6,5){$\frac 14$}
\rput(8,5){$\frac 15$}
\rput(10,5){$\cdots$}
\rput(1,4){$\frac 12$}
\rput(3,4){$\frac 1{12}$}
\rput(5,4){$\frac 1{36}$}
\rput(7,4){$\frac 1{80}$}
\rput(9,4){$\cdots$}
\rput(2,3){$\frac 5{12}$}
\rput(4,3){$\frac 1{36}$}
\rput(6,3){$\frac {11}{2160}$}
\rput(8,3){$\cdots$}
\rput(3,2){$\frac {7}{18}$}
\rput(5,2){$\frac {49}{4320}$}
\rput(7,2){$\cdots$}
\rput(4,1){$\frac {1631}{4320}$}
\rput(6,1){$\cdots$}

\end{pspicture}
\caption{Computing the sequence $\omega_n$}
\label{fige}
\end{table}
We see now that the modified Stirling cycle numbers may be expressed as linear combinations of the usual  Stirling cycle numbers: for $n,k \in \Z$ with $n\gqs 0$
\begin{align}
  \stira{n}{k}^* & = \sum_{j=1}^k \stira{0}{j}^* \stira{n}{k-j}\notag \\
  & = \sum_{j=1}^k (-1)^{j-1} \omega_j \stira{n}{k-j}.  \label{rr4}
\end{align}

The Akiyama-Tanigawa algorithm for Bernoulli numbers, as described in \cite{Kan}, is very similar to the  construction in Table \ref{fige}. With the same starting sequence $1$, $\frac 12$, $\frac 13$, $\frac 14, \dots$, it applies the recursion
\begin{equation}\label{tan2}
  a_{n+1,m}=(a_{n,m}-a_{n,m+1})\cdot(m+1) \qquad (n, m\gqs 0),
\end{equation}
multiplying by $m+1$ instead of dividing. The result is that $a_{n,0}$ is equal to the $n$th Bernoulli number $B_n$. 
 As these numbers have well-known important number-theoretic properties,  the similarly constructed sequence $\omega_n$ should also be of interest. Exploring numerically, we find that the denominator of $\omega_n$ factors into primes less than or equal to $n+1$.
The numerators of $\omega_n$ are (up to a sign) the integer sequence $A176599$ in the OEIS
\cite{OEIS}. These numerators have factorizations containing unusually large primes with, for example,
\begin{equation*}
  \text{numerator}(\omega_{20}) = 23 \cdot 29 \cdot 31 \cdot 37 \cdot 41^2 \cdot p, \qquad p \text{ prime} \approx 4.91473 \times 10^{73}.
\end{equation*}
Lastly we note that it appears to be the case that all the numbers $a_{n,m}$ in the array used to construct $\omega_n$ are positive with $\omega_n \to 1/e \approx 0.36788$ as $n \to \infty$.

\appendix
\section{Appendix: Second proof of Theorem \ref{ky}}
All the ideas in this appendix were outlined by the referee and we added some details to fill out the arguments.
Recalling Section \ref{Uu}, define the formal series
\begin{align*}
  f(t) & := \tfrac 12 (U(t)-u(t)) = c_1 t^{1/2}+ c_3 t^{3/2}+c_5 t^{5/2}+ \dots,\\
 g(t) & := \tfrac 12 (U(t)+u(t)) = c_0+ c_2 t +c_4 t^{2}+ \dots,
\end{align*}
with $\tilde{g}(t):=g(-t)$. Also put
\begin{equation*}
  h(t):=2 f'(t)-4 f'(t) g'(t) = U'(t)-u'(t) -\left(U'(t)^2-u'(t)^2 \right).
\end{equation*}
 The following result has a straightforward proof, with some of the computations appearing in Section \ref{fin}, and it shows that $f$, $g$ and $h$ can be used to give a generating function version of \e{ccj}.

\begin{prop} \label{topaa}
The identity \e{ccj} being true for all odd $n\gqs 3$ is equivalent to
\begin{equation}\label{ccjx}
  h(t)= -\tfrac 23 f'(t)-2(f'*\tilde{g}'')(t),
\end{equation}
as a formal power series identity, where $*$ denotes convolution.
\end{prop}

Here we wish to understand $f$, $g$ and $h$ as functions of $t$. Later, in \e{gaw}, \e{gaw2} and \e{gaw3}, it will be found that   the Laplace transforms of these functions essentially give $\G(n)$, $\theta_n$ and $\Psi_n$, respectively.
The plan to prove Theorem \ref{ky} via Proposition \ref{top} is  as follows. Both sides of \e{ccjx} are shown to be  smooth, slowly growing functions of $t>0$, though we do not know initially that they are equal. Their Laplace transforms are computed in Propositions \ref{gu1}, \ref{gu2}, \ref{gu3} and found to agree. It follows that both sides of \e{ccjx} do agree as functions of $t>0$.
Hence their expansions as $t\to 0^+$ are the same and this will prove the desired identity \e{ccj}.

\subsection{Properties of $U(t)$ and $u(t)$} \label{quk}
If we relate the real variables $t$ and $v$ by
\begin{equation*}
  e^{-t}=v e^{1-v},
\end{equation*}
then clearly $t=t(v)=v-\log v-1$ is a function of $v>0$ and may be differentiated any number of times. By the implicit function theorem, $v$ is also a function of $t$ in the two cases $v>1$ and $0<v<1$. We label these solutions $v=U(t)>1$ and $v=u(t)<1$. Then $U(t)$ and $u(t)$ may be differentiated as many times as $t(v)$ and so are smooth for $t>0$. We have the elementary estimates
\begin{equation}\label{ery}
  t+1 < U(t) < 2t+2, \qquad e^{-t-1} < u(t) < e^{-t} \qquad (t>0).
\end{equation}
By implicit differentiation,
\begin{equation*}
  U'(t)=1+1/(U(t)-1), \qquad u'(t)=1-1/(1-u(t)),
\end{equation*}
and hence, as $t \to \infty$, we find $U'(t) \to 1$ and $u'(t) \to 0$.

To understand $U(t)$ and $u(t)$ for small values of $t$, first recall the function $z V(z)$ from \e{vx} which may be seen to be holomorphic in a neighborhood of $z=0$ in $\C$. By the inverse function theorem for holomorphic functions, since the derivative of $z V(z)$ at $z=0$ is nonzero, it has an inverse that is holomorphic  in some neighborhood $\mathcal N$ of $0$. We labelled this inverse as $z V^*(z)$ in Sections \ref{ut} and \ref{Uu}.
Therefore
\begin{equation} \label{term}
  C(z):=1+\sqrt{2} z V^*(-\sqrt{2} z) = \sum_{j=0}^\infty c_j z^j,
\end{equation}
is holomorphic in  $\mathcal N$, with the shown Taylor expansion, and satisfies
\begin{equation*}
  e^{-z^2} = C(z) e^{1-C(z)}.
\end{equation*}
We must have $U(t)=C(\sqrt{t})$ and $u(t)=C(-\sqrt{t})$ for $t>0$, (see \e{Uux} below), and hence,
\begin{equation}\label{pluy}
  U(t) \sim \sum_{j=0}^\infty c_j t^{j/2}, \qquad u(t) \sim \sum_{j=0}^\infty (-1)^j c_j t^{j/2} \qquad \text{as $t\to 0^+$}.
\end{equation}
In the standard asymptotic expansion notation we are using, the left side of \e{pluy}, for example, means: for any integer $m\gqs 1$ there exist positive $\delta_m$ and $\kappa_m$ so that
\begin{equation*}
  \biggl| U(t) - \sum_{j=0}^{m-1} c_j t^{j/2} \biggr| \lqs \kappa_m t^{m/2}
\end{equation*}
for all $t$ with $0<t<\delta_m$.
In particular, since $c_0=1$ and $c_1=\sqrt{2}$,
\begin{equation} \label{Uux}
  U(t)=1+\sqrt{2} t^{1/2} +O(t), \qquad u(t)=1-\sqrt{2} t^{1/2} +O(t) \qquad \text{as $t\to 0^+$},
\end{equation}
and this confirms that $U(t)>1$ and $u(t)<1$ for small $t$ as  expected.

Also $C'(z)= \sum_{j=1}^\infty j c_j z^{j-1}$ holds  in $\mathcal N$. Then $U'(t)=C'(\sqrt{t})/(2\sqrt{t})$, and similarly for $u'(t)$, giving the expected derivatives of \e{pluy}:
\begin{equation}\label{pluy2}
  U'(t) \sim \sum_{j=1}^\infty \tfrac j2 c_j t^{j/2-1}, \qquad u'(t) \sim \sum_{j=1}^\infty (-1)^j \tfrac j2 c_j t^{j/2-1} \qquad \text{as $t\to 0^+$}.
\end{equation}
(In general it is not valid to differentiate asymptotic expansions; see \cite[p. 9]{Olv}.)

\subsection{Laplace transforms of $f$, $g$ and $h$}

The incomplete gamma functions are
\begin{equation*}
  \g(a,z):=\int_0^z e^{-t}t^{a-1}\, dt, \qquad \G(a,z):=\int_z^\infty e^{-t}t^{a-1}\, dt \qquad (\Re(a)>0),
\end{equation*}
with their properties listed in \cite[Chap. 8]{DLMF}. For $\Re(s)>0$ we obtain the Laplace transforms of $U$ and $u$ using the substitution $v-\log v-1=t$:
\begin{align*}
  \int_0^\infty U(t) e^{-st}\, dt  & = \int_1^\infty e^{-sv} v^s e^s (v-1)\, dv = \frac{e^s}{s^{s+2}}\G(s+1,s)+\frac 1s, \\
  \int_0^\infty u(t) e^{-st}\, dt  & = -\int_0^1 e^{-sv} v^s e^s (v-1)\, dv = -\frac{e^s}{s^{s+2}}\g(s+1,s)+\frac 1s.
\end{align*}
Letting $F(s)$ and $G(s)$ denote the Laplace transforms of $f$ and $g$ respectively then gives
\begin{equation}\label{fglap}
  F(s)  = \frac{e^s \G(s+1)}{2 s^{s+2}}, \qquad
  G(s)  = -\frac{e^s \G(s+1)}{2 s^{s+2}} + \frac{e^s \G(s+1,s)}{s^{s+2}} +\frac 1s.
\end{equation}

For the next result we will need Kummer's confluent hypergeometric function $M(a,b,z)$. This is described in \cite[Chap. 13]{DLMF} and has the alternate notation $_1F_1(a;b;z)$.

\begin{prop} \label{gu1}
Write $H(s)$ for the Laplace transform of $h$. For $\Re(s)>0$ with $s$ not an integer,
\begin{equation*}
  H(s) = \pi \cot(\pi s)-2s F(s) M(1,1-s,-s).
\end{equation*}
\end{prop}
\begin{proof}
From our bounds in Section \ref{quk} we have $h(t)=-\tfrac{\sqrt{2}}3 t^{-1/2}+O(t^{1/2})$ as $t\to 0^+$ and $h(t)=O(1)$ as $t\to \infty$ so that $H(s)$ is well defined and holomorphic for $\Re(s)>0$. However
\begin{equation*}
  U'(t)^2 =\frac 1{2t}+O(t^{-1/2}), \qquad  u'(t)^2 =\frac 1{2t}+O(t^{-1/2}) \qquad \text{as $t\to 0^+$},
\end{equation*}
so in breaking up the integral we require
\begin{equation*}
  H(s)= 2\int_0^\infty f'(t) e^{-st}\, dt + \lim_{\epsilon \to 0^+}\left[
 -\int_\epsilon^\infty U'(t)^2 e^{-st}\, dt + \int_\epsilon^\infty u'(t)^2 e^{-st}\, dt\right].
\end{equation*}
The Laplace transform of  $f'$ is $s F(s)-f(0)=s F(s)$, and with the substitution $v-\log v-1=t$,
\begin{equation*}
  H(s)= 2s F(s) + \lim_{\epsilon \to 0^+}\left[
 - \int_{1+\epsilon}^\infty \frac{e^{-s(v-1)} v^{s+1}}{v-1}\, dv +  \int_{1-\epsilon}^0 \frac{e^{-s(v-1)} v^{s+1}}{v-1}\, dv\right].
\end{equation*}
For any holomorphic $\phi$ we have
\begin{equation*}
   \lim_{\epsilon \to 0^+}\int_{1+\epsilon}^{1-\epsilon} \frac{\phi(v)}{v-1}\, dv =\pi i \cdot \phi(1),
\end{equation*}
with the path of integration  a semicircle of radius $\epsilon$ centered at $1$ with positive imaginary part.
Therefore
\begin{equation} \label{hs}
  H(s)= 2s F(s) - \pi i
 - e^s \int_{0}^\infty \frac{e^{-sv} v^{s+1}}{v-1}\, dv,
\end{equation}
for any path of integration that runs along the positive reals but moves above the pole at $v=1$.
With \cite[Eq. 8.6.9]{DLMF} and $\Re(s)>0$, $s\not\in \Z$, this integral satisfies
\begin{align}
   e^s \int_{0}^\infty \frac{e^{-sv} v^{s+1}}{v-1}\, dv & = -e^{\pi i s} \G(s+2) \G(-s-1, s e^{\pi i})\\
   & = -e^{\pi i s} \G(s+2) (\G(-s-1) - \g(-s-1, s e^{\pi i})) \notag\\
   & = -e^{\pi i s} \pi/\sin(\pi(s+2)) + e^{\pi i s} \G(s+2) \g(-s-1, s e^{\pi i}) \notag\\
    & = -\pi\cot(\pi s) -\pi i + e^{\pi i s} \G(s+2) \g(-s-1, s e^{\pi i}), \label{mmm}
\end{align}
where  the functional equation for $\G(s)$ was used, and $s e^{\pi i}$ keeps track of the branch we are on.
From \cite[Eq. 8.5.1]{DLMF},
\begin{equation*}
   \g(-s-1, s e^{\pi i}) = \frac{e^s e^{-\pi i s}}{(s+1) s^{s+1}} M(1,-s,-s).
\end{equation*}
Also $M(1,-s,-s)=M(1,1-s,-s)+1$ by \cite[Eq. 13.3.4]{DLMF}. Consequently, the last term in \e{mmm} can be written as
\begin{align*}
  e^{\pi i s} \G(s+2) \g(-s-1, s e^{\pi i}) & = e^{\pi i s} \G(s+2) \frac{e^s e^{-\pi i s}}{(s+1) s^{s+1}}(1+M(1,1-s,-s))\\
  & = 2s F(s)+2s F(s) M(1,1-s,-s).
\end{align*}
Combining our identities completes the proof.
\end{proof}

\subsection{The right side of \e{ccjx}}
First we need to give a meaning to $\tilde U(t):=U(-t)$ and $\tilde u(t):=u(-t)$ for $t>0$. This case was not considered by Watson in \cite{Wa29}. When $t>0$ we examine
\begin{equation*}
  e^{t}=v e^{1-v}
\end{equation*}
for complex $v$ satisfying $|\Im(v)|<\pi$. Write $v=x+iy=r e^{i \theta}$ for $x,y,r,\theta \in \R$ and $|y|<\pi$, $r > 0$. Then $t=1-v+\log v$ implies $t=1-x+\log r$ and $y=\theta$. Hence $\tan y = \tan \theta =y/x$ and we obtain
\begin{equation} \label{pth}
  x= y \cot y, \qquad t=\tfrac 12 \log\left( y^2 \cot^2 y +y^2\right)+1-y \cot y,
\end{equation}
even functions of $y$.
By the implicit function theorem we find $y$, (and hence $x$), as smooth functions of $t>0$ in the two cases $y>0$ and $y<0$. For definiteness we may set $\tilde U(t)=x+i y$ and $\tilde u(t)=x-i y$ for $y>0$.


\SpecialCoor
\psset{griddots=5,subgriddiv=0,gridlabels=0pt}
\psset{xunit=0.85cm, yunit=0.6cm}
\psset{linewidth=1pt}
\psset{dotsize=2pt 0,dotstyle=*}

\begin{figure}[ht]
\centering
\begin{pspicture}(-6,-3.6)(6,4.6) 

\psline[linecolor=gray](-6,0)(6,0)
\multirput(-5,-0.13)(1,0){11}{\psline[linecolor=gray](0,0)(0,0.26)}
\psline[linecolor=gray](0,-3.6)(0,4.6)
\multirput(-0.1,-3)(0,1){8}{\psline[linecolor=gray](0,0)(0.2,0)}

\rput(-5,-0.5){$_{-5}$}
\rput(-3,-0.5){$_{-3}$}
\rput(-1,-0.5){$_{-1}$}
\rput(5,-0.5){$_{5}$}
\rput(3,-0.5){$_{3}$}
\rput(1,-0.5){$_{1}$}

\rput(-0.4,-2){$_{-2}$}
\rput(-0.3,4){$_{4}$}
\rput(-0.3,2){$_{2}$}

\rput(1,2.4){$g(t)$}
\rput(1.5,1){$f(t)$}
\rput(-3,3){$\pm \Im f(t)$}

\savedata{\mydata}[
{{0., 0.}, {0.04, 0.283472}, {0.08, 0.40178}, {0.12, 0.49317}, {0.16,
  0.570727}, {0.2, 0.639505}, {0.24, 0.702093}, {0.28,
  0.760023}, {0.32, 0.814293}, {0.36, 0.865592}, {0.4,
  0.914424}, {0.44, 0.961166}, {0.48, 1.00611}, {0.52,
  1.04949}, {0.56, 1.09149}, {0.6, 1.13228}, {0.64, 1.17197}, {0.68,
  1.21067}, {0.72, 1.24848}, {0.76, 1.28549}, {0.8, 1.32175}, {0.84,
  1.35732}, {0.88, 1.39227}, {0.92, 1.42664}, {0.96, 1.46047}, {1.,
  1.4938}}, {1.1, 1.57514}, {1.2,
  1.65399}, {1.3, 1.7307}, {1.4, 1.80556}, {1.5, 1.8788}, {1.6,
  1.95061}, {1.7, 2.02114}, {1.8, 2.09054}, {1.9, 2.15892}, {2.,
  2.22639}, {2.1, 2.29302}, {2.2, 2.35889}, {2.3, 2.42407}, {2.4,
  2.48862}, {2.5, 2.55259}, {2.6, 2.61603}, {2.7, 2.67897}, {2.8,
  2.74146}, {2.9, 2.80352}, {3., 2.86519}, {3.1, 2.92649}, {3.2,
  2.98744}, {3.3, 3.04808}, {3.4, 3.10841}, {3.5, 3.16847}, {3.6,
  3.22825}, {3.7, 3.28779}, {3.8, 3.34709}, {3.9, 3.40617}, {4.,
  3.46503}, {4.1, 3.5237}, {4.2, 3.58217}, {4.3, 3.64047}, {4.4,
  3.69859}, {4.5, 3.75654}, {4.6, 3.81435}, {4.7, 3.872}, {4.8,
  3.92951}, {4.9, 3.98688}, {5., 4.04412}, {5.1, 4.10123}, {5.2,
  4.15822}, {5.3, 4.2151}, {5.4, 4.27186}, {5.5, 4.32852}, {5.6,
  4.38507}, {5.7, 4.44153}, {5.8, 4.49788}, {5.9, 4.55414}, {6.,
  4.61031}}
  ]
\dataplot[linecolor=orange,linewidth=0.8pt,plotstyle=line]{\mydata}

\savedata{\mydata}[
{{0,0}, {-0.000200004, 0.02}, {-0.000800071, 0.04}, {-0.00180036, 0.06},
{-0.00320114, 0.08}, {-0.00500278, 0.1}, {-0.00720577, 0.12},
{-0.00981069, 0.14}, {-0.0128182, 0.16}, {-0.0162292, 0.18},
{-0.0200446, 0.2}, {-0.0242654, 0.22}, {-0.0288926, 0.24},
{-0.0339277, 0.26}, {-0.0393719, 0.28}, {-0.0452268, 0.3},
{-0.0514939, 0.32}, {-0.0581751, 0.34}, {-0.065272, 0.36},
{-0.0727867, 0.38}, {-0.0807214, 0.4}, {-0.0890781, 0.42},
{-0.0978594, 0.44}, {-0.107068, 0.46}, {-0.116705, 0.48}, {-0.126776,
0.5}, {-0.153862, 0.55}, {-0.183719,
  0.6}, {-0.216403, 0.65}, {-0.251974, 0.7}, {-0.290504,
  0.75}, {-0.332068, 0.8}, {-0.376751, 0.85}, {-0.424649,
  0.9}, {-0.475864, 0.95}, {-0.530511, 1.}, {-0.588718,
  1.05}, {-0.650623, 1.1}, {-0.716383, 1.15}, {-0.786167,
  1.2}, {-0.860164, 1.25}, {-0.938586, 1.3}, {-1.02167,
  1.35}, {-1.10966, 1.4}, {-1.20287, 1.45}, {-1.3016, 1.5}, {-1.40623,
   1.55}, {-1.51717, 1.6}, {-1.63488, 1.65}, {-1.75988,
  1.7}, {-1.89277, 1.75}, {-2.03423, 1.8}, {-2.18506,
  1.85}, {-2.34616, 1.9}, {-2.51857, 1.95}, {-2.70355, 2.}, {-2.90252,
   2.05}, {-3.11722, 2.1}, {-3.34969, 2.15}, {-3.60241,
  2.2}, {-3.87839, 2.25}, {-4.18135, 2.3}, {-4.5159, 2.35}, {-4.88787,
   2.4}, {-5.30475, 2.45}, {-5.77629, 2.5}}
  ]
\dataplot[linecolor=orange,linewidth=0.8pt,plotstyle=line,linestyle=dashed]{\mydata}

\savedata{\mydata}[
{{0,0}, {-0.000200004, -0.02}, {-0.000800071, -0.04}, {-0.00180036, -0.06},
{-0.00320114, -0.08}, {-0.00500278, -0.1}, {-0.00720577, -0.12},
{-0.00981069, -0.14}, {-0.0128182, -0.16}, {-0.0162292, -0.18},
{-0.0200446, -0.2}, {-0.0242654, -0.22}, {-0.0288926, -0.24},
{-0.0339277, -0.26}, {-0.0393719, -0.28}, {-0.0452268, -0.3},
{-0.0514939, -0.32}, {-0.0581751, -0.34}, {-0.065272, -0.36},
{-0.0727867, -0.38}, {-0.0807214, -0.4}, {-0.0890781, -0.42},
{-0.0978594, -0.44}, {-0.107068, -0.46}, {-0.116705, -0.48},
{-0.126776, -0.5}, {-0.153862, -0.55}, {-0.183719, -0.6}, {-0.216403,
-0.65}, {-0.251974, -0.7}, {-0.290504, -0.75}, {-0.332068, -0.8},
{-0.376751, -0.85}, {-0.424649, -0.9}, {-0.475864, -0.95},
{-0.530511, -1.}, {-0.588718, -1.05}, {-0.650623, -1.1}, {-0.716383,
-1.15}, {-0.786167, -1.2}, {-0.860164, -1.25}, {-0.938586, -1.3},
{-1.02167, -1.35}, {-1.10966, -1.4}, {-1.20287, -1.45}, {-1.3016,
-1.5}, {-1.40623, -1.55}, {-1.51717, -1.6}, {-1.63488, -1.65},
{-1.75988, -1.7}, {-1.89277, -1.75}, {-2.03423, -1.8}, {-2.18506,
-1.85}, {-2.34616, -1.9}, {-2.51857, -1.95}, {-2.70355, -2.},
{-2.90252, -2.05}, {-3.11722, -2.1}, {-3.34969, -2.15}, {-3.60241,
-2.2}, {-3.87839, -2.25}, {-4.18135, -2.3}, {-4.5159, -2.35},
{-4.88787, -2.4}, {-5.30475, -2.45}, {-5.77629, -2.5}}
  ]
\dataplot[linecolor=orange,linewidth=0.8pt,plotstyle=line,linestyle=dashed]{\mydata}

\savedata{\mydata}[
{{-5.77629, -3.34662}, {-5.30475, -2.95888}, {-4.88787, -2.62005},
{-4.5159, -2.32106}, {-4.18135, -2.05501}, {-3.87839, -1.81653},
{-3.60241, -1.60137}, {-3.34969, -1.40617}, {-3.11722, -1.22818},
{-2.90252, -1.06518}, {-2.70355, -0.915315}, {-2.51857, -0.777055},
{-2.34616, -0.649107}, {-2.18506, -0.530381}, {-2.03423, -0.419946},
{-1.89277, -0.317007}, {-1.75988, -0.220877}, {-1.63488, -0.13096},
{-1.51717, -0.0467392}, {-1.40623, 0.032239}, {-1.3016,
  0.106372}, {-1.20287, 0.176012}, {-1.10966, 0.241467}, {-1.02167,
  0.303015}, {-0.938586, 0.3609}, {-0.860164, 0.415342}, {-0.786167,
  0.466535}, {-0.716383, 0.514657}, {-0.650623, 0.559865}, {-0.588718,
   0.602301}, {-0.530511, 0.642093}, {-0.475864,
  0.679356}, {-0.424649, 0.714196}, {-0.376751, 0.746706}, {-0.332068,
   0.776972}, {-0.290504, 0.80507}, {-0.251974, 0.831069}, {-0.216403,
   0.855033}, {-0.183719, 0.877018}, {-0.153862,
  0.897073}, {-0.126776, 0.915244}, {-0.10241, 0.931571}, {-0.0807214,
   0.946089}, {-0.0616714, 0.958829}, {-0.0452268,
  0.969818}, {-0.0313591, 0.979079}, {-0.0200446,
  0.986631}, {-0.0112641, 0.992489}, {-0.00500278,
  0.996664}, {-0.00125017, 0.999167},  {0., 1.}, {0.1,
  1.06652}, {0.2, 1.13274}, {0.3, 1.19868}, {0.4, 1.26433}, {0.5,
  1.32969}, {0.6, 1.39478}, {0.7, 1.45959}, {0.8, 1.52412}, {0.9,
  1.58839}, {1., 1.65239}, {1.1, 1.71614}, {1.2, 1.77962}, {1.3,
  1.84286}, {1.4, 1.90585}, {1.5, 1.9686}, {1.6, 2.03111}, {1.7,
  2.09338}, {1.8, 2.15543}, {1.9, 2.21725}, {2., 2.27886}, {2.1,
  2.34024}, {2.2, 2.40142}, {2.3, 2.4624}, {2.4, 2.52317}, {2.5,
  2.58375}, {2.6, 2.64413}, {2.7, 2.70433}, {2.8, 2.76435}, {2.9,
  2.82418}, {3., 2.88385}, {3.1, 2.94334}, {3.2, 3.00267}, {3.3,
  3.06183}, {3.4, 3.12084}, {3.5, 3.1797}, {3.6, 3.23841}, {3.7,
  3.29697}, {3.8, 3.35539}, {3.9, 3.41367}, {4., 3.47182}, {4.1,
  3.52983}, {4.2, 3.58772}, {4.3, 3.64548}, {4.4, 3.70312}, {4.5,
  3.76065}, {4.6, 3.81806}, {4.7, 3.87535}, {4.8, 3.93254}, {4.9,
  3.98962}, {5., 4.0466}, {5.1, 4.10348}, {5.2, 4.16026}, {5.3,
  4.21694}, {5.4, 4.27353}, {5.5, 4.33003}, {5.6, 4.38644}, {5.7,
  4.44276}, {5.8, 4.499}, {5.9, 4.55515}, {6., 4.61123}}
  ]
\dataplot[linecolor=black,linewidth=0.8pt,plotstyle=line]{\mydata}

\end{pspicture}
\caption{The graphs of $f(t)$ and $g(t)$ \label{fig}}
\end{figure}
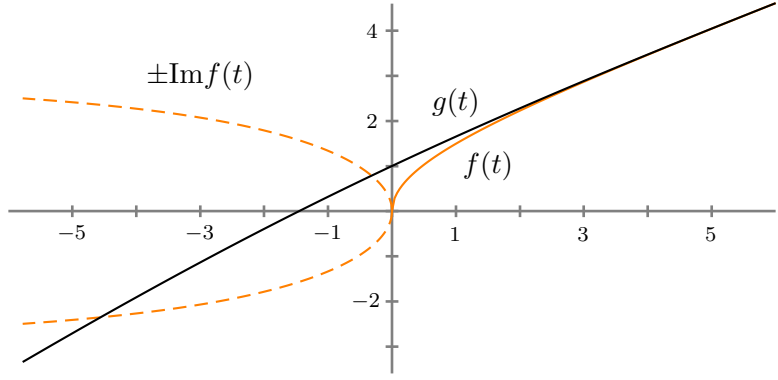


Our interest is in $\tilde{g}(t) := \tfrac 12 (\tilde U(t)+\tilde u(t)) = x$. For large enough $t$ we have the simple estimate
\begin{equation*}
  1-t < \tilde{g}(t) <1-\tfrac 12 t.
\end{equation*}
Also, since $\tilde U'(t)=-1+1/(1-\tilde U(t))$ and $\tilde u'(t)=-1+1/(1-\tilde u(t))$ by implicit differentiation, it follows that as $t \to \infty$ we have $\tilde{g}'(t) \to -1$ and $\tilde{g}''(t) \to 0$.

For small $t>0$ we must have $\tilde U(t) = C(i\sqrt{t})$ and $\tilde u(t) = C(-i\sqrt{t})$.
Therefore
\begin{equation*}
  \tilde{g}(t) \sim \sum_{j \gqs 0, \ j \ \text{even}} (-1)^{j/2}  c_j t^{j/2}
   \qquad \text{as $t\to 0^+$}.
\end{equation*}
In fact it may be seen that $g(t)$ (and  $\tilde{g}(t)$) are holomorphic in a neighborhood of $t=0$. (Since $\phi(z)=C(z)+C(-z)$ is holomorphic and even, we must have $\phi(\sqrt{z})$ well-defined and holomorphic.) It
follows in particular that
\begin{equation} \label{ner}
  \tilde{g}''(t) \sim \sum_{j \gqs 4, \ j \ \text{even}} (-1)^{j/2} \tfrac{j}{2}(\tfrac{j}{2} -1) c_j t^{j/2-2}
   \qquad \text{as $t\to 0^+$}.
\end{equation}
The graphs of $f(t)$ and $g(t)$ are shown in Figure \ref{fig} for $t$ positive and negative. When $t<0$, the function $f(t)=\tfrac 12 (U(t)-u(t))$ takes imaginary values.

\begin{prop} \label{gu2}
Write $\tilde G(s)$ for the Laplace transform of $\tilde g$. For $\Re(s)>0$ with $s$ not an integer,
\begin{equation*}
  \tilde G(s) = -e^{-s}s^{s-2}\cos(\pi s) \G(1-s) +\frac 1s -\frac 1{s^2}+ \frac 1{s^2} M(1,1-s,-s).
\end{equation*}
\end{prop}
\begin{proof}
From our bounds for $\tilde g(t)$, it is clear that $\tilde G(s)$ converges to a holomorphic function  for $\Re(s)>0$. As $\tilde g(t) = \Re (\tilde U(t))$, we have
\begin{equation*}
  \tilde G(s) = \int_0^\infty \tilde g(t) e^{-s t}\, dt = \Re \int_0^\infty \tilde U(t) e^{-s t}\, dt,
\end{equation*}
if $s$ is  a positive real, which is assumed from here on. Substituting $t=1-v+\log v$ finds
\begin{equation*}
  \int_0^\infty \tilde U(t) e^{-s t}\, dt = e^{-s}\int_{\mathcal C} e^{sv}v^{-s}(1-v)\,dv
\end{equation*}
where $\mathcal C$ is the curve in \e{pth}, parameterized by $y \cot y+i y$ for $0<y<\pi$, and going from $v=1$ to $-\infty+i\pi$. This integral over $\mathcal C$ may be computed by means of the antiderivative
\begin{equation*}
  \int e^{sv}v^{-s}(1-v)\,dv = \frac{e^{sv}}{s^2 v^s}\bigl( 1-sv-M(1,1-s,-sv)\bigr),
\end{equation*}
provided $s>0$ is not an integer. For $v$ approaching $-\infty+i\pi$, use  \cite[13.2.23]{DLMF} which says
\begin{equation*}
  M(a,b,z) \sim \frac{\G(b)}{\G(a)} e^z z^{a-b} \qquad \text{as $z \to \infty$}.
\end{equation*}
Precisely, the ratio of both sides tends to $1$ as $z \to \infty$ when $|\arg z|\lqs \pi/2-\delta$ for arbitrary fixed $\delta>0$. It follows that
\begin{equation*}
  \int_0^\infty \tilde U(t) e^{-s t}\, dt =
  -e^{-s}s^{s-2}e^{-\pi i s} \G(1-s) +\frac 1s -\frac 1{s^2} + \frac 1{s^2} M(1,1-s,-s).
\end{equation*}
Taking the real part of this, and then allowing all $s$ with $\Re(s)>0$ by analytic continuation, finishes the proof.
\end{proof}

\begin{prop} \label{gu3}
 For $\Re(s)>0$ with $s$ not an integer, the Laplace transform of the right side of \e{ccjx} is
\begin{equation*}
   \pi \cot(\pi s)-2s F(s) M(1,1-s,-s).
\end{equation*}
\end{prop}
\begin{proof}
By the well-known formulas for Laplace transforms of derivatives and convolutions, the right side of \e{ccjx}
becomes
\begin{equation*}
  -\tfrac 23 s F(s)-2 s F(s) \left(s^2 \tilde G(s) - s \tilde g(0) - \tilde g'(0) \right).
\end{equation*}
We have $\tilde g(0) =c_0 =1$ and $\tilde g'(0) =-c_2 =-2/3$. Inserting the formula from Proposition \ref{gu2} and simplifying completes the proof.
\end{proof}

\subsection{Proof of Theorem \ref{ky} and final remarks} \label{fin}

The next result describes the expansion of a convolution in the case we need.

\begin{lemma} \label{pwq}
Suppose $A(t)$ and $B(t)$ are continuous on some interval $0<t<\delta$ with
\begin{equation*}
  A(t) \sim \sum_{j=0}^\infty a_j t^{j-1/2}, \qquad B(t) \sim \sum_{k=0}^\infty b_k t^{k}  \qquad \text{as $t\to 0^+$}.
\end{equation*}
Then, as $t \to 0^+$,
\begin{equation*}
  (A * B)(t) \sim \sum_{n=0}^\infty d_n t^{n+1/2} \qquad \text{for} \qquad d_n= \sum_{j+k=n} \frac{\G(j+\tfrac 12) \G(k+1)}{\G(j+k+\tfrac 32)} a_j b_k.
\end{equation*}
\end{lemma}
\begin{proof}
It is straightforward to show that
\begin{equation*}
  (A * B)(t) := \int_0^t A(v) \cdot B(t-v)\, dv
\end{equation*}
has an expansion of the stated form. To compute $d_n$, a typical term is
\begin{align*}
  a_j b_k \int_0^t v^{j-1/2}  (t-v)^k\, dv & =   a_j b_k \cdot t^{j-1/2}\cdot t^k \cdot t \int_0^1 w^{j-1/2}  (1-w)^k\, dw \\
   & = a_j b_k \cdot t^{j+k+1/2} \frac{\G(j+\tfrac 12) \G(k+1)}{\G(j+k+\tfrac 32)},
\end{align*}
using the beta function.
\end{proof}

\begin{proof}[Alternate proof of Theorem \ref{ky}]
From Propositions \ref{gu1} and \ref{gu3}, the Laplace transforms of each side of \e{ccjx} give the same holomorphic function of $s$ when $\Re(s)>0$. It follows from Lerch's well-known theorem that the continuous functions on each side of \e{ccjx} must be the same for $t>0$.

Using \e{pluy2} we easily find, as $t \to 0^+$,
\begin{equation} \label{iu}
  h(t) \sim \sum_{n\gqs 3, \ n \text{ odd}} t^{n/2-2} \left( (n-2) c_{n-2} - \sum_{x+y=n} x c_x \cdot y c_y\right),
\end{equation}
where $x$ and $y$ will  mean odd and even nonnegative integers, respectively, for the rest of this proof.  By  \e{pluy2}, \e{ner} and Lemma \ref{pwq},
\begin{equation*}
  (f'*\tilde{g}'')(t) \sim \sum_{x\gqs 1, \ y\gqs 4} \frac{\G(\tfrac x2+1)\G(\tfrac y2+1)}{\G(\tfrac x2+\tfrac y2-1)} c_x \cdot (-1)^{y/2}  c_y  \cdot t^{(x+y)/2-2}.
\end{equation*}
The above sum may be extended to all $x$ and $y$ with $x+y\gqs 3$ provided the sums with $y=0$ and $y=2$ are subtracted. From the identity
\begin{equation*}
  \frac{\G(\tfrac x2+1)\G(\tfrac y2+1)}{\G(\tfrac x2+\tfrac y2-1)} = \frac{x!!y!!}{4(x+y-4)!!},
\end{equation*}
and the values $c_0=1$, $c_2=2/3$, we conclude that, as $t \to 0^+$,
\begin{multline} \label{iu2}
  -\tfrac 23 f'(t)-2(f'*\tilde{g}'')(t) \\
  \sim \sum_{n\gqs 3, \ n \text{ odd}} \frac{t^{n/2-2}}{2(n-4)!!} \left( n!! c_{n} - 2(n-2)!! c_{n-2}-\sum_{x+y=n} (-1)^{y/2} x!! c_x \cdot y!! c_y\right).
\end{multline}
The corresponding coefficients of \e{iu} and \e{iu2} must be equal, proving \e{ccj}  for all odd $n\gqs 3$. Theorem \ref{ky} now follows from Proposition \ref{top}.
\end{proof}

Lastly, we may relate the transforms $F(n)$, $G(n)$ and $H(n)$ with our objects of study $\G(n+1)$, $\theta_n$ and $\Psi_n$, and note a direct way to see their asymptotic expansions as $n \to \infty$.
We have already proved in \e{fglap} that
\begin{equation}\label{gaw}
  \G(n+1)=\frac{2 n^{n+2}}{e^n} F(n),
\end{equation}
for $n>0$. Watson's formula, from \cite[p. 294]{Wa29}, is
\begin{equation*}
  \theta_n=1-\frac n2 \int_{0}^\infty e^{-nt}\left(U'(t)+u'(t) \right)\, dt,
\end{equation*}
and integration   by parts then shows
\begin{equation}\label{gaw2}
  \theta_n=1+n-n^2 G(n).
\end{equation}
From \cite[Eq. (8.6)]{OSra},
\begin{equation*}
  \frac{e^n n!}{n^{n+1}} \Psi_n=-\pi i - e^n \int_{0}^\infty \frac{e^{-nv} v^n}{v-1}\, dv,
\end{equation*}
with the path of integration moving above $1$. Hence, by \e{hs},
\begin{equation}\label{gaw3}
 n! \Psi_n=\frac{n^{n+1}}{e^n } H(n).
\end{equation}

Now Watson's Lemma, in the form \cite[p.~113]{Olv} for example, gives the asymptotics of Laplace transforms such as $F(n)$, when $n \to \infty$, in terms of the expansions of the integrands $f(t)$ as $t \to 0^+$. In this way, with \e{gaw} and \e{pluy}, we obtain Stirling's approximation for $\G(n+1)$ in \e{gmx} and the relation on the right of \e{rg}.  With \e{gaw2}  we similarly obtain the asymptotic expansion of  $\theta_n$ in \e{iou} and the relation on the left of \e{rg}. 
Using \e{gaw3} and \e{iu},  the asymptotic expansion of  $n! \Psi_n$ in \e{gps} is proved along with  \e{rgg}.

{\small \bibliography{ram2-bib} }

\newcommand{\etalchar}[1]{$^{#1}$}
\begin{thebibliography}{GKP94}

\bibitem[AAR99]{AAR}
George~E. Andrews, Richard Askey, and Ranjan Roy.
\newblock {\em Special functions}, volume~71 of {\em Encyclopedia of
  Mathematics and its Applications}.
\newblock Cambridge University Press, Cambridge, 1999.

\bibitem[Ber89]{Ber89}
Bruce~C. Berndt.
\newblock {\em Ramanujan's notebooks. {P}art {II}}.
\newblock Springer-Verlag, New York, 1989.

\bibitem[BM11]{BM11}
Stella Brassesco and Miguel~A. M\'{e}ndez.
\newblock The asymptotic expansion for {$n!$} and the {L}agrange inversion
  formula.
\newblock {\em Ramanujan J.}, 24(2):219--234, 2011.

\bibitem[Com74]{Comtet}
Louis Comtet.
\newblock {\em Advanced combinatorics}.
\newblock D. Reidel Publishing Co., Dordrecht, enlarged edition, 1974.
\newblock The art of finite and infinite expansions.

\bibitem[{\relax DLMF}]{DLMF}
{\it NIST Digital Library of Mathematical Functions}.
\newblock http://dlmf.nist.gov/, Release 1.1.8 of 2022-12-15.
\newblock F.~W.~J. Olver, A.~B. {Olde Daalhuis}, D.~W. Lozier, B.~I. Schneider,
  R.~F. Boisvert, C.~W. Clark, B.~R. Miller, B.~V. Saunders, H.~S. Cohl, and
  M.~A. McClain, eds.

\bibitem[Fu]{Fu}
Amy~M. Fu.
\newblock Some identities related to the second-order {E}ulerian numbers.
\newblock arXiv: 2104.09316.

\bibitem[Ges16]{Ge16}
Ira~M. Gessel.
\newblock Lagrange inversion.
\newblock {\em J. Combin. Theory Ser. A}, 144:212--249, 2016.

\bibitem[GKP94]{Knu}
Ronald~L. Graham, Donald~E. Knuth, and Oren Patashnik.
\newblock {\em Concrete mathematics}.
\newblock Addison-Wesley Publishing Company, Reading, MA, second edition, 1994.
\newblock A foundation for computer science.

\bibitem[Kan00]{Kan}
Masanobu Kaneko.
\newblock The {A}kiyama-{T}anigawa algorithm for {B}ernoulli numbers.
\newblock {\em J. Integer Seq.}, 3(2):Article 00.2.9, 6, 2000.

\bibitem[Kar01]{Ka01}
Ekatherina~A. Karatsuba.
\newblock On the asymptotic representation of the {E}uler gamma function by
  {R}amanujan.
\newblock {\em J. Comput. Appl. Math.}, 135(2):225--240, 2001.

\bibitem[Mar86]{Mar86}
John C.~W. Marsaglia.
\newblock The incomplete gamma function and {R}amanujan's rational
  approximation to $e^x$.
\newblock {\em J. Stat. Comput. Simul.}, 24:163--168, 1986.

\bibitem[Mat]{mo}
Mathoverflow.
\newblock https://mathoverflow.net/questions/45756.

\bibitem[MM90]{mar2}
George Marsaglia and John C.~W. Marsaglia.
\newblock A new derivation of {S}tirling's approximation to {$n!$}.
\newblock {\em Amer. Math. Monthly}, 97(9):826--829, 1990.

\bibitem[Nem10]{nem}
Gerg\H{o} Nemes.
\newblock On the coefficients of the asymptotic expansion of {$n!$}.
\newblock {\em J. Integer Seq.}, 13(6):Article 10.6.6, 5, 2010.

\bibitem[Olv97]{Olv}
Frank W.~J. Olver.
\newblock {\em Asymptotics and special functions}.
\newblock AKP Classics. A K Peters, Ltd., Wellesley, MA, 1997.
\newblock Reprint of the 1974 original.

\bibitem[O'S]{OSra}
Cormac O'Sullivan.
\newblock Ramanujan's approximation to the exponential function and
  generalizations.
\newblock arXiv: 2205.08504.

\bibitem[O'S22]{odm}
Cormac O'Sullivan.
\newblock De {M}oivre and {B}ell polynomials.
\newblock {\em Expo. Math.}, 40(4):870--893, 2022.

\bibitem[Per17]{Pe17}
Oskar Perron.
\newblock \"{U}ber die n\"{a}herungsweise {B}erechnung von {F}unktionen
  gro{\ss}er {Z}ahlen.
\newblock {\em Sitzungsber. Bayr. Akad. Wissensch. (M\"unch. Ber.)}, pages
  191--219, 1917.

\bibitem[RU19]{RU19}
Grzegorz Rz\c{a}dkowski and Ma{\l}gorzata Urli\'{n}ska.
\newblock Some applications of the generalized {E}ulerian numbers.
\newblock {\em J. Combin. Theory Ser. A}, 163:85--97, 2019.

\bibitem[S{\etalchar{+}}22]{OEIS}
N.~J.~A. Sloane et~al.
\newblock {\em The {O}n-{L}ine {E}ncyclopedia of {I}nteger {S}equences}.
\newblock Published electronically at https://oeis.org, 2022.

\bibitem[Vol08]{Vo08}
Hans Volkmer.
\newblock Factorial series connected with the {L}ambert function, and a problem
  posed by {R}amanujan.
\newblock {\em Ramanujan J.}, 16(3):235--245, 2008.

\bibitem[Wat29]{Wa29}
G.~N. Watson.
\newblock Theorems stated by {R}amanujan ({V}): {A}pproximations connected with
  $e^x$.
\newblock {\em Proc. London Math. Soc. (2)}, 29(4):293--308, 1929.

\bibitem[Wre68]{Wr68}
John~W. Wrench, Jr.
\newblock Concerning two series for the gamma function.
\newblock {\em Math. Comp.}, 22:617--626, 1968.

\end{thebibliography}

{\small 
\vskip 5mm
\noindent
\textsc{Dept. of Math, The CUNY Graduate Center, 365 Fifth Avenue, New York, NY 10016-4309, U.S.A.}

\noindent
{\em E-mail address:} \texttt{cosullivan@gc.cuny.edu}
}

\end{document}

{\small 
\vskip 5mm
\noindent
\textsc{Dept. of Math, The CUNY Graduate Center, 365 Fifth Avenue, New York, NY 10016-4309, U.S.A.}

\noindent
{\em E-mail address:} \texttt{cosullivan@gc.cuny.edu}
}

\begin{lemma}

\end{lemma}
\begin{proof}

\end{proof}

\begin{prop}

\end{prop}
\begin{proof}

\end{proof}

\begin{theorem}

\end{theorem}
\begin{proof}

\end{proof}

\sum_{\substack{0\leqslant h<k  \\ (h,k)=1}}